\documentclass[11pt,reqno,twoside]{amsart}
\usepackage{amssymb,amsmath,amsthm,soul,color,paralist}
\usepackage{t1enc}
\usepackage[cp1250]{inputenc}
\usepackage{a4,indentfirst,latexsym}
\usepackage{graphics}
\usepackage{mathrsfs}
\usepackage{cite,enumitem,graphicx}
\usepackage[colorlinks=true,urlcolor=blue,
citecolor=red,linkcolor=blue,linktocpage,pdfpagelabels,
bookmarksnumbered,bookmarksopen]{hyperref}
\usepackage[english]{babel}
\usepackage[left=2.50cm,right=2.50cm,top=2.72cm,bottom=2.72cm]{geometry}
\usepackage[metapost]{mfpic}
\usepackage[hyperpageref]{backref}
\usepackage[colorinlistoftodos]{todonotes}
\usepackage[normalem]{ulem}

\makeatletter
\providecommand\@dotsep{5}
\def\listtodoname{List of Todos}
\def\listoftodos{\@starttoc{tdo}\listtodoname}
\makeatother

\numberwithin{equation}{section}

\newcommand{\R}{\mathbb{R}}
\newcommand{\N}{\mathcal{N}}

\newcommand{\C}{\mathbb{C}}

\newcommand{\h}{H^{s}_{\varepsilon}}
\newcommand{\2}{2^{*}_{s}}

\DeclareMathOperator{\dive}{div}
\DeclareMathOperator{\supp}{supp}
\DeclareMathOperator{\e}{\varepsilon}

\newtheorem{lem}{Lemma}[section]
\newtheorem{thm}{Theorem}[section]
\newtheorem{defn}{Definition}[section]

\newtheorem{remark}{Remark}[section]

\title[fractional Schr\"odinger equations with magnetic field]{Existence and concentration results for some fractional Schr\"odinger equations  in $\R^{N}$ with magnetic fields}

\author[V. Ambrosio]{Vincenzo Ambrosio}
\address{Vincenzo Ambrosio\hfill\break\indent 
Dipartimento di Ingegneria Industriale e Scienze Matematiche \hfill\break\indent
Universit\`a Politecnica delle Marche\hfill\break\indent
Via Brecce Bianche, 12\hfill\break\indent
60131 Ancona (Italy)}
\email{ambrosio@dipmat.univpm.it}


\begin{document}

\maketitle

\begin{abstract}
We consider some nonlinear fractional Schr\"odinger equations with magnetic field and involving continuous nonlinearities having subcritical, critical or supercritical growth. Under a local condition on the potential, we use minimax methods to investigate the existence and concentration of nontrivial weak solutions. 
\end{abstract}

\maketitle

\section{introduction}

\noindent 
In the first part of this paper we study the following nonlinear fractional Schr\"odinger equation
\begin{equation}\label{P}
\e^{2s}(-\Delta)_{A/\e}^{s}u+V(x)u=f(|u|^{2})u \quad \mbox{ in } \R^{N},
\end{equation}
where $\e>0$ is a parameter, $s\in (0,1)$, $N\geq 3$ and $A:\R^{N}\rightarrow \R^{N}$ is a  $C^{0,\alpha}$, with $\alpha\in(0,1]$, magnetic potential.
The magnetic fractional  Laplacian $(-\Delta)^{s}_{A}$ is defined, up to a normalization constant, for all $u\in C^{\infty}_{c}(\R^{N}, \C)$ by setting
\begin{equation}\label{operator}
(-\Delta)^{s}_{A}u(x)
:= 2 \lim_{r\rightarrow 0} \int_{\R^{N}\setminus B_{r}(x)} \frac{u(x)-e^{\imath A(\frac{x+y}{2})\cdot (x-y)} u(y)}{|x-y|^{N+2s}} dy.
\end{equation}
This operator has been introduced in \cite{DS, I10} and relies essentially on the L\'evy-Khintchine formula
for the generator of a general L\'evy process.
From a physical point of view, when $s=\frac{1}{2}$, 
the operator in \eqref{operator} takes inspiration from the definition of a quantized operator corresponding to the classical relativistic Hamiltonian symbol for a relativistic particle of mass $m\geq 0$, that is
$$
\sqrt{(\xi-A(x))^{2}+m^{2}}+V(x), \quad (\xi, x)\in \R^{N}\times \R^{N},
$$
which is the sum of the kinetic energy term involving magnetic vector potential $A(x)$ and the potential energy term of electric scalar potential $V(x)$. 
For the sake of completeness, we emphasized that in the literature there are three kinds of quantum relativistic Hamiltonians depending on how to quantize the kinetic energy term $\sqrt{(\xi-A(x))^{2}+m^{2}}$.
As explained in \cite{I10}, these three non-local operators are in general different from each other but coincide when the vector potential $A$ is assumed to be linear, so in particular, in the case of constant magnetic fields.
For a more detailed description of the consistence of the definition of $(-\Delta)^{s}_{A}$ and for some recent results established for problems involving this operator, we refer the interested reader to \cite{AD, DS, FPV, PSV2, PSV1, ZSZ} and the references therein.\\
When $s\rightarrow 1$, equation \eqref{P} is related to the study of solutions $u:\R^{N}\rightarrow \C$ of the following nonlinear Schr\"odinger equation with magnetic field
\begin{equation}\label{MSE}
\left(\frac{\e}{\imath}\nabla-A(x)\right)^{2} u+V(x)u=f(|u|^{2})u \quad \mbox{ in } \R^{N},
\end{equation}
where 
$\left(\frac{\e}{\imath}\nabla-A\right)^{2}$ is the magnetic Laplacian given by
$$
\left(\frac{\e}{\imath}\nabla-A\right)^{2}\!u= -\e^{2}\Delta u -\frac{2\e}{\imath} A(x) \cdot \nabla u + |A(x)|^{2} u -\frac{\e}{\imath} u \dive(A(x)).
$$ 
In this context, when $N=3$, the magnetic field $B$ is exactly the curl of $A$, while  for higher dimensions $N\geq 4$, $B$ should be thought of as a 2-form given by $B_{ij} =\partial_{j}A_{k}-\partial_{k}A_{j}$; see \cite{AHS, RS}. \\
Equation \eqref{MSE} arises in the investigation of standing wave solutions $\psi(x, t)=u(x)e^{-\imath \frac{E}{\e}t}$, with $E\in \R$, for the following time-dependent nonlinear Schr\"odinger equation 
$$
\imath \e \frac{\partial \psi}{\partial t}=\left(\frac{\e}{\imath}\nabla-A(x)\right)^{2} \psi+W(x)\psi-f(|\psi|^{2})\psi \quad \mbox{ in } (x, t)\in \R^{N}\times \R
$$
where $W(x)=V(x)+E$. An important class of solutions of \eqref{MSE} are the so called semi-classical states which concentrate and develop a spike shape around one, or more, particular points in $\R^{N}$, while vanishing elsewhere as $\e\rightarrow 0$. This interest is due to the well-known fact that the transition from Quantum Mechanics to Classical Mechanics can be formally performed by sending $\e\rightarrow 0$.
For this reason, equation \eqref{MSE} has been widely studied by many authors \cite{EL, AF, AFF, AS, Cingolani, Cingolani-Secchi, K}.\\
In the nonlocal framework, if the vector field $A\equiv 0$, problem \eqref{P} reduces to a fractional Schr\"odinger equation of the type
\begin{equation}\label{FSE}
\e^{2s}(-\Delta)^{s}u+V(x)u=f(u) \quad \mbox{ in } \R^{N},
\end{equation}
introduced by Laskin \cite{Laskin1} as a fundamental equation of fractional Quantum Mechanics in the study of particles on stochastic fields modeled by L\'evy processes.
In the recent literature, several existence and multiplicity results for \eqref{FSE} have been established by applying different variational and topological approaches: see for instance \cite{AM, A1, A6, DDPW, DPMV, FQT, FigS, HZ, MBRS, Secchi}.\\
The potential $V: \R^{N}\rightarrow \R$ appearing in \eqref{P} is a continuous function verifying the following conditions due to del Pino and Felmer \cite{DF}:
\begin{compactenum}[$(V_1)$]
\item $\inf_{x\in \R^{N}} V(x)=V_{1}>0$;
\item there exists a bounded open set $\Lambda\subset \R^{N}$ such that
$$
0<V_{0}=\inf_{x\in \Lambda} V(x)<\min_{x\in \partial \Lambda} V(x).
$$
\end{compactenum}
We note that no restriction on the global behavior of $V$ is required: in particular, $V$ is not required to be bounded or to belong to a Kato class.

\noindent
Concerning the nonlinearity $f: \R\rightarrow \R$, we assume that $f$ is continuous, $f(t)=0$ for $t\leq 0$ and satisfies the following assumptions:
\begin{compactenum}[$(f_1)$]
\item $\displaystyle{\lim_{t\rightarrow 0} f(t)=0}$;
\item there exists $q\in (2, 2^{*}_{s})$, where $2^{*}_{s}=2N/(N-2s)$, such that $\lim_{t\rightarrow \infty} f(t)/t^{\frac{q-2}{2}}=0$;
\item there exists $\theta>2$ such that $0<\frac{\theta}{2} F(t)\leq t f(t)$ for any $t>0$, where $F(t)=\int_{0}^{t} f(\tau)d\tau$;
\item  $f(t)$ is increasing for $t>0$.
\end{compactenum} 
In a recent paper \cite{AD}, the author and d'Avenia established the existence and multiplicity of nontrivial solutions to \eqref{P}, for $\e>0$ small, requiring  that $V$ verifies the global condition introduced by Rabinowitz \cite{Rab}:
$$
\liminf_{|x|\rightarrow \infty} V(x)>\inf_{x\in \R^{N}} V(x).
$$
Their results have been strongly influenced by the work \cite{AFF} in which the authors dealt with \eqref{MSE} under local assumptions on the potential $V$.\\
Motivated by \cite{AFF, AD, DF}, in this paper we focus our attention on the existence and concentration of weak nontrivial solutions to \eqref{P} by supposing that $V$ satisfies $(V_{1})$-$(V_{2})$. For simplicity, we will assume that $0\in \Lambda$ and $V_{1}=V_{0}=V(0)$.

The first main result of this paper is the following:
\begin{thm}\label{thm1}
Suppose that $V$ satisfies $(V_1)$-$(V_2)$ and $f$ satisfies $(f_1)$-$(f_4)$. Then there exists $\e_{0}>0$ such that, for any $\e\in (0, \e_{0})$, problem \eqref{P} has a nontrivial solution $u_{\e}$. Moreover, if $\eta_{\e}\in \R^{N}$ is the global maximum point of  $|u_{\e}|$, we have that
$$
\lim_{\e\rightarrow 0} V(\eta_{\e})=V_{0},
$$
and there exists $C>0$ such that
$$
|u_{\e}(x)|\leq \frac{C\e^{N+2s}}{\e^{N+2s}+|x-\eta_{\e}|^{N+2s}} \quad \forall x\in \R^{N}.
$$
\end{thm}
The proof of Theorem \ref{thm1} is obtained by using suitable variational methods.  More precisely, inspired by \cite{AFF, DF}, we modify the nonlinearity $f$ outside the set $\Lambda$ in such way that the energy functional of the modified problem satisfies the Palais-Smale condition (see Lemma \ref{PSc}). 
In order to prove that the solutions of the modified problem also satisfy \eqref{P} for $\e>0$ small enough, we use in an appropriate way a Moser iteration scheme \cite{Moser} and some recent results established in \cite{FQT}. A similar approach, combined with the extension method \cite{CS}, has been brilliantly used in \cite{AM} to study the existence and concentration of positive solutions for the fractional Schr\"odinger equation \eqref{FSE}. Anyway, when $A\neq 0$, we can not directly adapt these techniques due to the presence of the magnetic fractional Laplacian $(-\Delta)^{s}_{A}$. Moreover, in the fractional magnetic case, the estimates on the modulus of solutions are more delicate. Therefore, a more careful analysis is essential to prove that the (translated) sequence $(u_{n})$ of solutions of the modified problem verifies the property $|u_{n}(x)|\rightarrow 0$ as $|x|\rightarrow \infty$ uniformly with respect to $n\in \mathbb{N}$. We give a sketch of our idea. Firstly we prove that  each  $|u_{n}|$ is bounded in $L^{\infty}(\R^{N}, \R)$ uniformly in $n\in \mathbb{N}$, by means of a Moser iteration argument. At this point, we would like to use a fractional magnetic Kato's inequality \cite{Kato} to deduce that each $|u_{n}|$ verifies
\begin{equation}\label{AAV}
(-\Delta)^{s}|u_{n}|+V_{0}|u_{n}|\leq g(\e x, |u_{n}|^{2})|u_{n}| \mbox{ in } \R^{N}.
\end{equation}
If this were true, then we can exploit a comparison argument (see at the end of Lemma \ref{moser} below) and the results in \cite{FQT} to deduce informations on the decay  at infinity of each $|u_{n}|$. We believe that a Kato's inequality is available for $(-\Delta)^{s}_{A}$ but we are not able to prove it except for rough functions which are bounded from below and above (see Remark $3.1$). Anyway, in order to show that each $|u_{n}|$ solves \eqref{AAV}, we use $\displaystyle{\frac{u_{n}}{u_{\delta,n}}\varphi}$ as test function in the modified problem
, where $u_{\delta,n}=\sqrt{|u_{n}|^{2}+\delta^{2}}$ and $\varphi$ is a real smooth nonnegative function with compact support in $\R^{N}$, and then we take the limit as $\delta\rightarrow 0$. 
We point out that our approach is completely different from \cite{AFF} in which the authors only use a suitable Moser iteration argument to prove that the solutions of the modified problem are also solutions of the original one.
However, the iteration in \cite{AFF} does not seem to be easy to adapt in our framework. 
Finally, we also establish a power-type decay estimate for $|u_{n}|$ which is in clear accordance with the results obtained in \cite{FQT}.

The second part of this paper deals with the following critical problem
\begin{equation}\label{Pcritico}
\e^{2s}(-\Delta)_{A/\e}^{s}u+V(x)u=f(|u|^{2})u+|u|^{2^{*}_{s}-2}u \quad \mbox{ in } \R^{N},
\end{equation}
where $f$ satisfies the following assumptions:
\begin{compactenum}[$(h_1)$]
\item $\displaystyle{\lim_{t\rightarrow 0} f(t)=0}$;
\item there exist $C_{0}>0$ and $q, \sigma\in (2, 2^{*}_{s})$ such that 
\begin{align*}
f(t)\geq C_{0} t^{\frac{q-2}{2}} \mbox{ for all } t\geq 0 \, \mbox{ and } \lim_{t\rightarrow \infty} f(t)/t^{\frac{\sigma-2}{2}}=0;
\end{align*}
\item there exists $\theta\in (2, \sigma)$ such that $0<\frac{\theta}{2} F(t)\leq t f(t)$ for any $t>0$, where $F(t)=\int_{0}^{t} f(\tau)d\tau$;
\item  $f(t)$ is increasing for $t>0$.
\end{compactenum} 

This time we have an extra difficulty in the study of our problem which is due to the presence of the critical exponent. Anyway, we will show that the approach developed to study the subcritical case can be adapted, after suitable modifications, to the critical one. Clearly, the calculations performed to get compactness are much more involved than those of the previous case, and we make use of the Concentration-Compactness Lemma for the fractional Laplacian \cite{DPMV, PP}; see proof of Lemma \ref{PSccritico}.
Our second main result can be stated as follows:
\begin{thm}\label{thm2}
Suppose that $V$ satisfies $(V_1)$-$(V_2)$ and $f$ satisfies $(h_1)$-$(h_4)$. 
Then there exists $\e_{0}>0$ such that, for any $\e\in (0, \e_{0})$, problem \eqref{P} has a nontrivial solution. Moreover, if $\eta_{\e}\in \R^{N}$ is the global maximum point of  $|u_{\e}|$, we have that
$$
\lim_{\e\rightarrow 0} V(\eta_{\e})=V_{0},
$$
and there exists $C>0$ such that
$$
|u_{\e}(x)|\leq \frac{C\e^{N+2s}}{\e^{N+2s}+|x-\eta_{\e}|^{N+2s}} \quad \forall x\in \R^{N}.
$$
\end{thm}

In the last part of this paper we consider a supercritical version of problem \eqref{Pcritico}. More precisely, we are concerned with the following problem:
\begin{equation}\label{Psupercritico}
\e^{2s}(-\Delta)^{s}_{A_{\e}} u + V(x)u =  |u|^{q-2}u+\lambda |u|^{r-2}u \mbox{ in } \R^{N},
\end{equation}
where $\e>0$, $\lambda>0$, and $2<q<2^{*}_{s}< r$. In this case we are able to prove that:
\begin{thm}\label{thm3}
Suppose that $V$ verifies $(V_1)$-$(V_2)$. Then there exists $\lambda_{0}>0$ with the following property: for any $\lambda\in (0, \lambda_{0})$ there exists $\e_{\lambda}>0$ such that, for any $\e\in (0, \e_{\lambda})$, problem  \eqref{Psupercritico} has a nontrivial solution.
Moreover, if $\eta_{\e}\in \R^{N}$ is the global maximum of $|u_{\e}|$, then 
$$
\lim_{\e\rightarrow 0} V(\eta_{\e})=V_{0}.
$$
\end{thm}
The proof of Theorem \eqref{thm3} relies on the truncated technique used in \cite{CY, FF, R}. Indeed, when we deal with the supercritical exponent, we can not directly use variational techniques because the corresponding functional is not well-defined on the fractional Sobolev space $\h$ (see Section $2$ for its definition).
To overcome this difficulty, we consider a truncated problem with subcritical growth and applying Theorem \ref{thm1} we obtain the existence of a nontrivial weak solution for the truncated problem. After proving a priori bounds (independent of $\lambda$) for this solution, we use a suitable Moser iteration scheme to verify that the the solution of the truncated problem is indeed a solution of the original problem \eqref{Psupercritico} provided that the parameter $\lambda$ is sufficiently small. 

We would like to emphasize that our results complement and extend in nonlocal magnetic framework the ones in \cite{AFF},  in the sense that we are considering the existence and concentration of nontrivial weak solutions for fractional magnetic problems in the whole space with subcritical, critical and supercritical continuous nonlinearities. In fact, the results presented here seem to be new also in the case $s=1$. 
Moreover, to our knowledge, this is the first time that the penalization technique is used to study fractional problems with magnetic fields. 

The plan of the paper is the following. In Section $2$ we recall some useful results regarding the functional setting. In Section $3$ we provide the proof of Theorem \ref{thm1}. In Section $4$ we focus our attention on the existence of nontrivial solutions to \eqref{Pcritico}. The Section $5$ is devoted to the supercritical problem \eqref{Psupercritico}.

\section{preliminaries and functional setup}
We begin recalling some definitions and results which will be useful along the paper; see \cite{AD, DS} for more details.

Let us denote by $L^{2}(\R^{N}, \C)$ the space of complex-valued functions with summable square, endowed with the real
scalar product 
$$
\langle u, v\rangle_{L^{2}}=\Re\left(\int_{\R^{N}} u \bar{v} dx\right)
$$
for all $u, v\in L^{2}(\R^{N}, \C)$.
We consider the space
$$
\mathcal{D}^{s}_{A}(\R^{N}, \C)=\{u\in L^{2^{*}_{s}}(\R^{N}, \C) : [u]_{s,A}<\infty\}
$$
where
$$
[u]_{s,A}^{2}=\iint_{\R^{2N}} \frac{|u(x)-u(y)e^{\imath A(\frac{x+y}{2})\cdot (x-y)}|^{2}}{|x-y|^{N+2s}} dx dy.
$$
Then, we define the following fractional magnetic Sobolev space
$$
H^{s}_{A}(\R^{N}, \C)=\{u\in L^{2}(\R^{N}, \C): [u]_{s,A}<\infty\}.
$$
Clearly, $H^{s}_{A}(\R^{N}, \C)$ is a Hilbert space with the real scalar product
$$
\langle u, v\rangle_{s, A}=\langle u, v\rangle_{L^{2}}+\Re\iint_{\R^{2N}} \frac{(u(x)-u(y)e^{\imath A(\frac{x+y}{2})\cdot (x-y)})\overline{(v(x)-v(y)e^{\imath A(\frac{x+y}{2})\cdot (x-y)})}}{|x-y|^{N+2s}} dx dy
$$ 
for any $u, v\in H^{s}_{A}(\R^{N}, \C)$. Moreover, $C^{\infty}_{c}(\R^{N}, \C)$ is dense in $H^{s}_{A}(\R^{N}, \C)$ (see Lemma $2.2$ in \cite{AD}).
\begin{thm}\label{Sembedding}\cite{DS}
The space $H^{s}_{A}(\R^{N}, \C)$ is continuously embedded into $L^{r}(\R^{N}, \C)$ for any $r\in [2, \2]$ and compactly embedded into $L^{r}(K, \C)$ for any $r\in [1, \2)$ and any compact $K\subset \R^{N}$.
\end{thm}

\begin{lem}\label{DI}\cite{DS}
For any $u\in H^{s}_{A}(\R^{N}, \C)$, we get $|u|\in H^{s}(\R^{N},\R)$ and it holds
$$
[|u|]_{s}\leq [u]_{s,A},
$$
where $$[v]^{2}_{s}=\iint_{\R^{2N}} \frac{|v(x)-v(y)|^{2}}{|x-y|^{N+2s}} \, dx dy$$ denotes the Gagliardo seminorm of a real valued function $v:\R^{N}\rightarrow \R$.\\
We also have the following pointwise diamagnetic inequality 
$$
||u(x)|-|u(y)||\leq |u(x)-u(y)e^{\imath A(\frac{x+y}{2})\cdot (x-y)}| \mbox{ a.e. } x, y\in \R^{N}.
$$
\end{lem}

\begin{remark}
Since $s\in (0, 1)$ is fixed, in order to simplify the notation, we will write $[\cdot]$ and $[\cdot]_{A}$ to denote  $[\cdot]_{s}$ and  $[\cdot]_{s,A}$, respectively. 
\end{remark}

\begin{lem}\label{aux}\cite{AD}
If $u\in H^{s}(\R^{N}, \R)$ and $u$ has compact support, then $w=e^{\imath A(0)\cdot x} u \in H^{s}_{A}(\R^{N}, \C)$.
\end{lem}


Using the change of variable $u(x)\mapsto u(\e x)$, we can see that (\ref{P}) is equivalent to the following problem
\begin{equation}\label{R}
(-\Delta)^{s}_{A_{\e}} u + V_{\e}(x)u =  f(|u|^{2})u \mbox{ in } \R^{N},
\end{equation}
where $A_{\e}(x):=A(\e x)$ and $V_{\e}(x):=V(\e x)$. \\
Fix $k>\frac{\theta}{\theta-2}$ and $a>0$ such that $f(a)=\frac{V_{0}}{k}$, and we introduce the functions
$$
\tilde{f}(t):=
\begin{cases}
f(t)& \text{ if $t \leq a$} \\
\frac{V_{0}}{k}    & \text{ if $t >a$}
\end{cases}
$$ 
and
$$
g(x, t)=\chi_{\Lambda}(x)f(t)+(1-\chi_{\Lambda}(x))\tilde{f}(t),
$$
where $\chi_{\Lambda}$ is the characteristic function on $\Lambda$, and  we write $G(x, t)=\int_{0}^{t} g(x, \tau)\, d\tau$.\\
From assumptions $(f_1)$-$(f_4)$, it follows that $g$ verifies the following properties:
\begin{compactenum}[($g_1$)]
\item $\displaystyle{\lim_{t\rightarrow 0} g(x, t)=0}$ uniformly in $x\in \R^{N}$;
\item $\lim_{t\rightarrow \infty} \frac{g(x,t)}{t^{\frac{q-2}{2}}}=0$ uniformly in $x\in \R^{N}$;
\item $(i)$ $0< \frac{\theta}{2} G(x, t)\leq g(x, t)t$ for any $x\in \Lambda$ and $t>0$, \\
$(ii)$ $0\leq  G(x, t)\leq g(x, t)t\leq \frac{V(x)}{k}t$ for any $x\in \R^{N}\setminus \Lambda$ and $t>0$;
\item $t\mapsto g(x,t)$ is increasing for $t>0$.
\end{compactenum}

Let us consider the following auxiliary problem 
\begin{equation}\label{Pe}
(-\Delta)^{s}_{A_{\e}} u + V_{\e}(x)u =  g_{\e}(x, |u|^{2})u \mbox{ in } \R^{N}, 
\end{equation}
where $g_{\e}(x, t):=g(\e x, t)$. 
Let us note that if $u$ is a solution of (\ref{Pe}) such that 
\begin{equation}\label{ue}
|u(x)|\leq a \mbox{ for all } x\in  \R^{N}\setminus \Lambda_{\e},
\end{equation}
where $\Lambda_{\e}:=\{x\in \R^{N}: \e x\in \Lambda\}$, then $u$ is also a solution of the original problem  (\ref{R}).

It is clear that weak solutions to (\ref{Pe}) can be found as critical points of the Euler-Lagrange functional
\begin{align*}
J_{\e}(u)
=\frac{1}{2}\|u\|^{2}_{\e}-\frac{1}{2}\int_{\R^{N}} G_{\e}(x, |u|^{2})\, dx
\end{align*}
which is well-defined for any function $u: \R^{N}\rightarrow \C$ belonging to the space
$$
\h=\left\{u\in \mathcal{D}^{s}_{A_{\e}}(\R^{N}, \C): \int_{\R^{N}} V_{\e}(x) |u|^{2}\, dx<\infty\right\}
$$
endowed with the norm 
$$
\|u\|^{2}_{\e}:=[u]^{2}_{A_{\e}}+\|\sqrt{V_{\e}} |u|\|^{2}_{L^{2}(\R^{N})}.
$$
We also consider the autonomous problem associated with \eqref{Pe}, that is
\begin{equation}\label{APe}
(-\Delta)^{s} u + V_{0}u =  f(u^{2})u \mbox{ in } \R^{N}, 
\end{equation}
and we denote by $I_{0}: H^{s}(\R^{N}, \R)\rightarrow \R$ the corresponding energy functional
\begin{align*}
I_{0}(u)
=\frac{1}{2}\|u\|^{2}_{0}-\frac{1}{2}\int_{\R^{N}} F(u^{2})\, dx
\end{align*}
where we used the notation $\|u\|_{0}^{2}:=[u]^{2}+\|\sqrt{V_{0}} |u|\|^{2}_{L^{2}(\R^{N})}$ which is a norm in $H^{s}(\R^{N}, \R)$ equivalent to the standard one.

In what follows, we show that $J_{\e}$ verifies the assumptions of the mountain pass theorem \cite{AR}. 
\begin{lem}\label{MPG}
\begin{compactenum}[$(i)$]
\item $J_{\e}(0)=0$;
\item there exist $\alpha, \rho>0$ such that $J_{\e}(u)\geq \alpha$ for any $u\in \h$ such that $\|u\|_{\e}=\rho$;
\item there exists $e\in \h$ with $\|e\|_{\e}>\rho$ such that $J_{\e}(e)<0$.
\end{compactenum}
\end{lem}
\begin{proof}
Using $(g_1)$-$(g_2)$ and Theorem \ref{Sembedding}, we can see that for any $\delta>0$ there exists $C_{\delta}>0$ such that
$$
J_{\e}(u)\geq \frac{1}{2}\|u\|^{2}_{\e}-\delta \|u\|^{2}_{\e}-C_{\delta} \|u\|^{q}_{\e}.
$$
Choosing $\delta>0$ sufficiently small, we can conclude that $(i)$ holds.
Regarding $(ii)$, we can note that in view of $(g_3)$, for any $u\in \h\setminus\{0\}$ with $supp(u)\subset \Lambda_{\e}$ and $t>0$ we have
\begin{align*}
J_{\e}(tu)&\leq \frac{t^{2}}{2} \|u\|^{2}_{\e}-\frac{1}{2}\int_{\Lambda_{\e}} G_{\e}(x, t^{2}|u|^{2})\, dx \\
&\leq \frac{t^{2}}{2} \|u\|^{2}_{\e}-Ct^{\theta} \int_{\Lambda_{\e}} |u|^{\theta}\, dx+C
\end{align*}
which implies that $J_{\e}(tu)\rightarrow -\infty$ as $t\rightarrow \infty$.
\end{proof}

\begin{lem}\label{PSc}
Let $c\in \R$. Then $J_{\e}$ satisfies the Palais-Smale condition at the level $c$.
\end{lem}
\begin{proof}
Let $(u_{n})\subset \h$ be a $(PS)_{c}$ sequence. Then $(u_{n})$ is bounded. Indeed, using $(g_3)$, we have
\begin{align*}
c+o_{n}(1)\|u_{n}\|_{\e}&\geq J_{\e}(u_{n})-\frac{1}{\theta}\langle J'_{\e}(u_{n}), u_{n}\rangle \\
&\geq \left(\frac{1}{2}-\frac{1}{\theta}\right)\|u_{n}\|^{2}_{\e}+\frac{1}{\theta}\int_{\R^{N}\setminus \Lambda_{\e}} \left[g_{\e}(x, |u_{n}|^{2})|u_{n}|^{2}\,-\frac{\theta}{2} G_{\e}(x, |u_{n}|^{2})\right]\, dx\\
&\geq \frac{1}{2}\left(\frac{\theta-2}{\theta}-\frac{1}{k}\right)\|u_{n}\|^{2}_{\e},
\end{align*}
and recalling that $k>\frac{\theta}{\theta-2}$ we get the thesis.
Now, we show that for any $\xi>0$ there exists $R=R_{\xi}>0$ such that 
\begin{equation}\label{T}
\limsup_{n\rightarrow \infty} \int_{\R^{N}\setminus B_{R}} \int_{\R^{N}}  \frac{|u_{n}(x)-u_{n}(y)e^{\imath A_{\e}(\frac{x+y}{2})\cdot (x-y)}|^{2}}{|x-y|^{N+2s}} dx dy+\int_{\R^{N}\setminus B_{R}} V_{\e}(x)|u_{n}|^{2}\, dx\leq \xi.
\end{equation}
Assume that the above claim is true and we show how it can be used to conclude the proof of lemma.
We know that $u_{n}\rightharpoonup u$ in $\h$. Since $\h \Subset L^{r}_{loc}(\R^{N}, \C)$ and $g$ has subcritical growth, it is easy to prove that $J'_{\e}(u)=0$. In particular,
$$
\|u\|^{2}_{\e}=\int_{\R^{N}} g_{\e}(x, |u|^{2})|u|^{2}\, dx.
$$
Recalling that $\langle J'_{\e}(u_{n}), u_{n}\rangle =o_{n}(1)$, we can infer that 
$$
\|u_{n}\|^{2}_{\e}=\int_{\R^{N}} g_{\e}(x, |u_{n}|^{2})|u_{n}|^{2}\, dx+o_{n}(1).
$$
Therefore, using the above claim, Theorem \ref{Sembedding} and $(g_1)$-$(g_2)$ we can conclude that
$$
\lim_{n\rightarrow \infty} \int_{\R^{N}} g_{\e}(x, |u_{n}|^{2})|u_{n}|^{2}\, dx=\int_{\R^{N}} g_{\e}(x, |u|^{2})|u|^{2}\, dx
$$
which yields
$$
\lim_{n\rightarrow \infty}\|u_{n}\|^{2}_{\e}=\|u\|^{2}_{\e}.
$$
Now we show \eqref{T}. Let $\eta_{R}\in C^{\infty}(\R^{N}, \R)$ be such that $0\leq \eta_{R}\leq 1$, $\eta_{R}=0$ in $B_{\frac{R}{2}}$, $\eta_{R}=1$ in $\R^{N}\setminus B_{R}$ and $|\nabla \eta_{R}|\leq \frac{C}{R}$ for some $C>0$ independent of $R$.
Since $\langle J'_{\e}(u_{n}), \eta_{R}u_{n}\rangle =o_{n}(1)$ we have
\begin{align*}
&\Re\left(\iint_{\R^{2N}} \frac{(u_{n}(x)-u_{n}(y)e^{\imath A_{\e}(\frac{x+y}{2})\cdot (x-y)})\overline{(u_{n}(x)\eta_{R}(x)-u_{n}(y)\eta_{R}(y)e^{\imath A_{\e}(\frac{x+y}{2})\cdot (x-y)})}}{|x-y|^{N+2s}}\, dx dy \right)\\
&\quad +\int_{\R^{N}} V_{\e}(x)\eta_{R} |u_{n}|^{2}\, dx=\int_{\R^{N}} g_{\e}(x, |u_{n}|^{2})|u_{n}|^{2}\eta_{R}\, dx+o_{n}(1).
\end{align*}
Fix $R>0$ such that $\Lambda_{\e}\subset B_{R/2}$. Taking into account that
\begin{align*}
&\Re\left(\iint_{\R^{2N}} \frac{(u_{n}(x)-u_{n}(y)e^{\imath A_{\e}(\frac{x+y}{2})\cdot (x-y)})\overline{(u_{n}(x)\eta_{R}(x)-u_{n}(y)\eta_{R}(y)e^{\imath A_{\e}(\frac{x+y}{2})\cdot (x-y)})}}{|x-y|^{N+2s}}\, dx dy \right)\\
&=\Re\left(\iint_{\R^{2N}} \overline{u_{n}(y)}e^{-\imath A_{\e}(\frac{x+y}{2})\cdot (x-y)}\frac{(u_{n}(x)-u_{n}(y)e^{\imath A_{\e}(\frac{x+y}{2})\cdot (x-y)})(\eta_{R}(x)-\eta_{R}(y))}{|x-y|^{N+2s}}  \,dx dy\right)\\
&\quad +\iint_{\R^{2N}} \eta_{R}(x)\frac{|u_{n}(x)-u_{n}(y)e^{\imath A_{\e}(\frac{x+y}{2})\cdot (x-y)}|^{2}}{|x-y|^{N+2s}}\, dx dy,
\end{align*}
and using $(g_3)$-$(ii)$, we have
\begin{align}\label{PS1}
&\iint_{\R^{2N}} \eta_{R}(x)\frac{|u_{n}(x)-u_{n}(y)e^{\imath A_{\e}(\frac{x+y}{2})\cdot (x-y)}|^{2}}{|x-y|^{N+2s}}\, dx dy+\int_{\R^{N}} V_{\e}(x)\eta_{R} |u_{n}|^{2}\, dx\nonumber\\
&\leq -\Re\left(\iint_{\R^{2N}} \overline{u_{n}(y)}e^{-\imath A_{\e}(\frac{x+y}{2})\cdot (x-y)}\frac{(u_{n}(x)-u_{n}(y)e^{\imath A_{\e}(\frac{x+y}{2})\cdot (x-y)})(\eta_{R}(x)-\eta_{R}(y))}{|x-y|^{N+2s}}  \,dx dy\right) \nonumber\\
&\quad +\frac{1}{k}\int_{\R^{N}} V_{\e}(x) \eta_{R} |u_{n}|^{2}\, dx+o_{n}(1).
\end{align}
From the H\"older inequality and the boundedness of $(u_{n})$ in $\h$ it follows that
\begin{align}\label{PS2}
&\left|\Re\left(\iint_{\R^{2N}} \overline{u_{n}(y)}e^{-\imath A_{\e}(\frac{x+y}{2})\cdot (x-y)}\frac{(u_{n}(x)-u_{n}(y)e^{\imath A_{\e}(\frac{x+y}{2})\cdot (x-y)})(\eta_{R}(x)-\eta_{R}(y))}{|x-y|^{N+2s}}  \,dx dy\right)\right| \nonumber\\
&\leq \left(\iint_{\R^{2N}} \frac{|u_{n}(x)-u_{n}(y)e^{\imath A_{\e}(\frac{x+y}{2})\cdot (x-y)}|^{2}}{|x-y|^{N+2s}}\,dxdy  \right)^{\frac{1}{2}} \left(\iint_{\R^{2N}} |\overline{u_{n}(y)}|^{2}\frac{|\eta_{R}(x)-\eta_{R}(y)|^{2}}{|x-y|^{N+2s}} \, dxdy\right)^{\frac{1}{2}} \nonumber\\
&\leq C \left(\iint_{\R^{2N}} |u_{n}(y)|^{2}\frac{|\eta_{R}(x)-\eta_{R}(y)|^{2}}{|x-y|^{N+2s}} \, dxdy\right)^{\frac{1}{2}}.
\end{align}
Now we prove that
\begin{equation}\label{PS3}
\limsup_{R\rightarrow \infty}\limsup_{n\rightarrow \infty} \iint_{\R^{2N}} |u_{n}(y)|^{2}\frac{|\eta_{R}(x)-\eta_{R}(y)|^{2}}{|x-y|^{N+2s}} \, dxdy=0.
\end{equation}
Firstly, we note that  
$$
\R^{2N}=((\R^{N}\setminus B_{2R})\times (\R^{N}\setminus B_{2R})) \cup ((\R^{N}\setminus B_{2R})\times B_{2R})\cup (B_{2R}\times \R^{N})=: X^{1}_{R}\cup X^{2}_{R} \cup X^{3}_{R}.
$$
Consequently,
\begin{align}\label{Pa1}
&\iint_{\R^{2N}}\frac{|\eta_{R}(x)-\eta_{R}(y)|^{2}}{|x-y|^{N+2s}} |u_{n}(x)|^{2} dx dy =\iint_{X^{1}_{R}}\frac{|\eta_{R}(x)-\eta_{R}(y)|^{2}}{|x-y|^{N+2s}} |u_{n}(x)|^{2} dx dy \nonumber \\
&\quad +\iint_{X^{2}_{R}}\frac{|\eta_{R}(x)-\eta_{R}(y)|^{2}}{|x-y|^{N+2s}} |u_{n}(x)|^{2} dx dy+
\iint_{X^{3}_{R}}\frac{|\eta_{R}(x)-\eta_{R}(y)|^{2}}{|x-y|^{N+2s}} |u_{n}(x)|^{2} dx dy.
\end{align}
Since $\eta_{R}=1$ in $\R^{N}\setminus B_{2R}$, we can see that
\begin{align}\label{Pa2}
\iint_{X^{1}_{R}}\frac{|u_{n}(x)|^{2}|\eta_{R}(x)-\eta_{R}(y)|^{2}}{|x-y|^{N+2s}} dx dy=0.
\end{align}
Now, fix $K>4$. Then
\begin{equation*}
X^{2}_{R}=(\R^{N} \setminus B_{2R})\times B_{2R} \subset ((\R^{N}\setminus B_{KR})\times B_{2R})\cup ((B_{KR}\setminus B_{2R})\times B_{2R}). 
\end{equation*}
Let us note that if $(x, y) \in (\R^{N}\setminus B_{KR})\times B_{2R}$ then
\begin{equation*}
|x-y|\geq |x|-|y|\geq |x|-2R>\frac{|x|}{2}. 
\end{equation*}
Therefore, using the above inequality, $0\leq \eta_{R}\leq 1$, $|\nabla \eta_{R}|\leq \frac{C}{R}$ and applying the H\"older inequality we obtain
\begin{align}\label{Pa3}
&\iint_{X^{2}_{R}}\frac{|u_{n}(x)|^{2}|\eta_{R}(x)-\eta_{R}(y)|^{2}}{|x-y|^{N+2s}} dx dy \nonumber \\
&=\int_{\R^{N}\setminus B_{KR}} \int_{B_{2R}} \frac{|u_{n}(x)|^{2}|\eta_{R}(x)-\eta_{R}(y)|^{2}}{|x-y|^{N+2s}} dx dy + \int_{B_{KR}\setminus B_{2R}} \int_{B_{2R}} \frac{|u_{n}(x)|^{2}|\eta_{R}(x)-\eta_{R}(y)|^{2}}{|x-y|^{N+2s}} dx dy \nonumber \\
&\leq C \int_{\R^{N}\setminus B_{KR}} \int_{B_{2R}} \frac{|u_{n}(x)|^{2}}{|x|^{N+2s}}\, dxdy+ \frac{C}{R^{2}} \int_{B_{KR}\setminus B_{2R}} \int_{B_{2R}} \frac{|u_{n}(x)|^{2}}{|x-y|^{N+2(s-1)}}\, dxdy \nonumber \\
&\leq CR^{N} \int_{\R^{N}\setminus B_{KR}} \frac{|u_{n}(x)|^{2}}{|x|^{N+2s}}\, dx + \frac{C}{R^{2}} (KR)^{2(1-s)} \int_{B_{KR}\setminus B_{2R}} |u_{n}(x)|^{2} dx \nonumber \\
&\leq CR^{N} \left( \int_{\R^{N}\setminus B_{KR}} |u_{n}(x)|^{2^{*}_{s}} dx \right)^{\frac{2}{2^{*}_{s}}} \left(\int_{\R^{N}\setminus B_{KR}}\frac{1}{|x|^{\frac{N^{2}}{2s} +N}}\, dx \right)^{\frac{2s}{N}} + \frac{C K^{2(1-s)}}{R^{2s}} \int_{B_{KR}\setminus B_{2R}} |u_{n}(x)|^{2} dx \nonumber \\
&\leq \frac{C}{K^{N}} \left( \int_{\R^{N}\setminus B_{KR}} |u_{n}(x)|^{2^{*}_{s}} dx \right)^{\frac{2}{2^{*}_{s}}} + \frac{C K^{2(1-s)}}{R^{2s}} \int_{B_{KR}\setminus B_{2R}} |u_{n}(x)|^{2} dx \nonumber \\
&\leq \frac{C}{K^{N}}+ \frac{C K^{2(1-s)}}{R^{2s}} \int_{B_{KR}\setminus B_{2R}} |u_{n}(x)|^{2} dx,
\end{align}
for some constant $C>0$ independent of $n$.
Take $\e\in (0,1)$ and we have
\begin{align}\label{Ter1}
&\iint_{X^{3}_{R}} \frac{|u_{n}(x)|^{2} |\eta_{R}(x)- \eta_{R}(y)|^{2}}{|x-y|^{N+2s}}\, dxdy \nonumber\\
&\leq \int_{B_{2R}\setminus B_{\varepsilon R}} \int_{\R^{N}} \frac{|u_{n}(x)|^{2} |\eta_{R}(x)- \eta_{R}(y)|^{2}}{|x-y|^{N+2s}}\, dxdy + \int_{B_{\varepsilon R}} \int_{\R^{N}} \frac{|u_{n}(x)|^{2} |\eta_{R}(x)- \eta_{R}(y)|^{2}}{|x-y|^{N+2s}}\, dxdy. 
\end{align}
Since
\begin{align*}
\int_{B_{2R}\setminus B_{\varepsilon R}} \int_{\R^{N} \cap \{y: |x-y|<R\}} \frac{|u_{n}(x)|^{2} |\eta_{R}(x)- \eta_{R}(y)|^{2}}{|x-y|^{N+2s}}\, dxdy \leq \frac{C}{R^{2s}} \int_{B_{2R}\setminus B_{\varepsilon R}} |u_{n}(x)|^{2} dx
\end{align*}
and 
\begin{align*}
\int_{B_{2R}\setminus B_{\varepsilon R}} \int_{\R^{N} \cap \{y: |x-y|\geq R\}} \frac{|u_{n}(x)|^{2} |\eta_{R}(x)- \eta_{R}(y)|^{2}}{|x-y|^{N+2s}}\, dxdy \leq \frac{C}{R^{2s}} \int_{B_{2R}\setminus B_{\varepsilon R}} |u_{n}(x)|^{2} dx,
\end{align*}
we get
\begin{align}\label{Ter2}
\int_{B_{2R}\setminus B_{\varepsilon R}} \int_{\R^{N}} \frac{|u_{n}(x)|^{2} |\eta_{R}(x)- \eta_{R}(y)|^{2}}{|x-y|^{N+2s}}\, dxdy \leq \frac{C}{R^{2s}} \int_{B_{2R}\setminus B_{\varepsilon R}} |u_{n}(x)|^{2} dx. 
\end{align}
On the other hand, from the definition of $\eta_{R}$, $\e\in (0,1)$, and $0\leq \eta_{R}\leq 1$ we obtain
\begin{align}\label{Ter3}
\int_{B_{\varepsilon R}} \int_{\R^{N}} \frac{|u_{n}(x)|^{2} |\eta_{R}(x)- \eta_{R}(y)|^{2}}{|x-y|^{N+2s}}\, dxdy &= \int_{B_{\varepsilon R}} \int_{\R^{N}\setminus B_{R}} \frac{|u_{n}(x)|^{2} |\eta_{R}(x)- \eta_{R}(y)|^{2}}{|x-y|^{N+2s}}\, dxdy\nonumber \\
&\leq C \int_{B_{\varepsilon R}} \int_{\R^{N}\setminus B_{R}} \frac{|u_{n}(x)|^{2}}{|x-y|^{N+2s}}\, dxdy\nonumber \\
&\leq C \int_{B_{\varepsilon R}} |u_{n}|^{2} dx \int_{(1-\e)R}^{\infty} \frac{1}{r^{1+2s}} dr\nonumber \\
&=\frac{C}{[(1-\e)R]^{2s}} \int_{B_{\varepsilon R}} |u_{n}|^{2} dx,
\end{align}
where we used the fact that if $(x, y) \in B_{\varepsilon R}\times (\R^{N} \setminus B_{R})$ then $|x-y|>(1-\e)R$. \\
Putting together \eqref{Ter1}, \eqref{Ter2} and \eqref{Ter3} we have
\begin{align}\label{Pa4}
\iint_{X^{3}_{R}} &\frac{|u_{n}(x)|^{2} |\eta_{R}(x)- \eta_{R}(y)|^{2}}{|x-y|^{N+2s}}\, dxdy \nonumber \\
&\leq \frac{C}{R^{2s}} \int_{B_{2R}\setminus B_{\varepsilon R}} |u_{n}(x)|^{p} dx + \frac{C}{[(1-\e)R]^{2s}} \int_{B_{\varepsilon R}} |u_{n}(x)|^{2} dx. 
\end{align}
In light of \eqref{Pa1}, \eqref{Pa2}, \eqref{Pa3} and \eqref{Pa4} we can infer 
\begin{align}\label{Pa5}
&\iint_{\R^{2N}} \frac{|u_{n}(x)|^{2} |\eta_{R}(x)- \eta_{R}(y)|^{2}}{|x-y|^{N+2s}}\, dxdy \nonumber \\
&\leq \frac{C}{K^{N}} + \frac{CK^{2(1-s)}}{R^{2s}} \int_{B_{KR}\setminus B_{2R}} |u_{n}(x)|^{2} dx + \frac{C}{R^{2s}} \int_{B_{2R}\setminus B_{\varepsilon R}} |u_{n}(x)|^{2} dx + \frac{C}{[(1-\e)R]^{2s}}\int_{B_{\varepsilon R}} |u_{n}(x)|^{2} dx. 
\end{align}
Since $(|u_{n}|)$ is bounded in $H^{s}(\R^{N}, \R)$, and using the compact Sobolev embedding $H^{s}(\R^{N}, \R)\subset L^{2}_{loc}(\R^{N}, \R)$ (see \cite{DPV}), we may assume that $|u_{n}|\rightarrow |u|$ in $L^{2}_{loc}(\R^{N}, \R)$ for some $u\in H^{s}(\R^{N}, \R)$. Then, taking the limit as $n\rightarrow \infty$ in \eqref{Pa5} we get
\begin{align*}
&\limsup_{n\rightarrow \infty} \iint_{\R^{2N}} \frac{|u_{n}(x)|^{2} |\eta_{R}(x)- \eta_{R}(y)|^{2}}{|x-y|^{N+2s}}\, dxdy\\
&\leq \frac{C}{K^{N}} + \frac{CK^{2(1-s)}}{R^{2s}} \int_{B_{KR}\setminus B_{2R}} |u(x)|^{2} dx + \frac{C}{R^{2s}} \int_{B_{2R}\setminus B_{\varepsilon R}} |u(x)|^{2} dx + \frac{C}{[(1-\e)R]^{2s}}\int_{B_{\varepsilon R}} |u(x)|^{2} dx \\
&\leq \frac{C}{K^{N}} + CK^{2} \left( \int_{B_{KR}\setminus B_{2R}} |u(x)|^{2^{*}_{s}} dx\right)^{\frac{2}{2^{*}_{s}}} + C\left(\int_{B_{2R}\setminus B_{\varepsilon R}} |u(x)|^{2^{*}_{s}} dx\right)^{\frac{2}{2^{*}_{s}}} + C\left( \frac{\e}{1-\e}\right)^{2s} \left(\int_{B_{\varepsilon R}} |u(x)|^{2^{*}_{s}} dx\right)^{\frac{2}{2^{*}_{s}}}, 
\end{align*}
where in the last passage we used the H\"older inequality. 
From $u\in L^{2^{*}_{s}}(\R^{N}, \R)$, $K>4$ and $\e \in (0,1)$ it follows that
\begin{align*}
\limsup_{R\rightarrow \infty} \int_{B_{KR}\setminus B_{2R}} |u(x)|^{2^{*}_{s}} dx = \limsup_{R\rightarrow \infty} \int_{B_{2R}\setminus B_{\varepsilon R}} |u(x)|^{2^{*}_{s}} dx = 0, 
\end{align*}
and taking $\e= \frac{1}{K}$ we get
\begin{align*}
&\limsup_{R\rightarrow \infty} \limsup_{n\rightarrow \infty} \iint_{\R^{2N}} \frac{|u_{n}(x)|^{2} |\eta_{R}(x)- \eta_{R}(y)|^{2}}{|x-y|^{N+2s}}\, dxdy\\
&\leq \lim_{K\rightarrow \infty} \limsup_{R\rightarrow \infty} \Bigl[\, \frac{C}{K^{N}} + CK^{2} \left( \int_{B_{KR}\setminus B_{2R}} |u(x)|^{2^{*}_{s}} dx\right)^{\frac{2}{2^{*}_{s}}} + C\left(\int_{B_{2R}\setminus B_{\frac{1}{K} R}} |u(x)|^{2^{*}_{s}} dx\right)^{\frac{2}{2^{*}_{s}}} \\
&\quad + C\left(\frac{1}{K-1}\right)^{2s} \left(\int_{B_{\frac{1}{K} R}} |u(x)|^{2^{*}_{s}} dx\right)^{\frac{2}{2^{*}_{s}}}\, \Bigr]\\
&\leq \lim_{k\rightarrow \infty} \frac{C}{K^{N}} + C\left(\frac{1}{K-1}\right)^{2s} \left(\int_{\R^{N}} |u(x)|^{2^{*}_{s}} dx \right)^{\frac{2}{2^{*}_{s}}}= 0.
\end{align*}
In conclusion, we proved that \eqref{PS3} is verified.  
Then, putting together \eqref{PS1}, \eqref{PS2} and \eqref{PS3} we obtain that
$$
\limsup_{R\rightarrow \infty}\limsup_{n\rightarrow \infty} \int_{\R^{N}\setminus B_{R}} \int_{\R^{N}}  \frac{|u_{n}(x)-u_{n}(y)e^{\imath A_{\e}(\frac{x+y}{2})\cdot (x-y)}|^{2}}{|x-y|^{N+2s}} dx dy+\left(1-\frac{1}{k}\right) \int_{\R^{N}\setminus B_{R}} V_{\e}(x)|u_{n}|^{2}\, dx=0
$$
which implies that \eqref{T} holds true.
\end{proof}

\noindent
Taking into account Lemma \ref{MPG}, we can define the mountain pass level
$$
c_{\e}=\inf_{\gamma\in \Gamma_{\e}} \max_{t\in [0, 1]} J_{\e}(\gamma(t))
$$
where
$$
\Gamma_{\e}=\{\gamma\in C([0, 1], \h): \gamma(0)=0 \mbox{ and } J_{\e}(\gamma(1))<0\}.
$$
Applying the mountain pass theorem \cite{AR}, we can see that there exists $u_{\e}\in \h\setminus\{0\}$ such that $J_{\e}(u_{\e})=c_{\e}$ and $J'_{\e}(u_{\e})=0$. 
Let us now introduce the Nehari manifold associated with (\ref{Pe}), namely
\begin{equation*}
\mathcal{N}_{\e}:= \{u\in \h \setminus \{0\} : \langle J_{\e}'(u), u \rangle =0\}.
\end{equation*}
It is standard to verify that $c_{\e}$ can be characterized as follows:
$$
c_{\e}=\inf_{u\in \h\setminus\{0\}} \sup_{t\geq 0} J_{\e}(t u)=\inf_{u\in \N_{\e}} J_{\e}(u);
$$
see \cite{W} for more details. 
Similarly, one can prove that $I_{0}$ has a mountain pass geometry, and denoting by $\mathcal{N}_{0}$ the Nehari manifold associated with \eqref{APe}, we obtain that  $c_{0}:=\inf_{\mathcal{N}_{0}} I_{0}$ coincides with the mountain pass level of $I_{0}$.
Next, we prove a very interesting relation between $c_{\e}$ and $c_{0}$. 

\begin{lem}\label{AMlem1}
The numbers $c_{\e}$ and $c_{0}$ satisfy the following inequality
$$
\limsup_{\e\rightarrow 0} c_{\e}\leq c_{0}.
$$
\end{lem}
\begin{proof}
Let $w\in H^{s}(\R^{N}, \R)$ be a positive ground state to the autonomous problem \eqref{APe}, so that $I'_{0}(w)=0$ and $I_{0}(w)=c_{0}$, and let $\eta\in C^{\infty}_{c}(\R^{N}, [0,1])$ be a cut-off function such that $\eta=1$ in $B_{\frac{\delta}{2}}$ and $\supp(\eta)\subset B_{\delta}\subset \Lambda$ for some $\delta>0$. 
We recall that the existence of $w$ is guaranteed in view of the results in \cite{Aade, FQT, FigS}. Moreover, from \cite{FQT}, we know that $w\in C^{0, \gamma}(\R^{N}, \R)$, for some $\gamma>0$, and 
\begin{equation}\label{remdecay}
0<w(x)\leq \frac{C}{|x|^{N+2s}} \mbox{ for all } |x|>1.
\end{equation}

Let us define $w_{\e}(x):=\eta_{\e}(x)w(x) e^{\imath A(0)\cdot x}$, with $\eta_{\e}(x)=\eta(\e x)$ for $\e>0$, and we observe that $|w_{\e}|=\eta_{\e}w$ and $w_{\e}\in \h$ in light of Lemma \ref{aux}. Let us prove that
\begin{equation}\label{limwr}
\lim_{\e\rightarrow 0}\|w_{\e}\|^{2}_{\e}=\|w\|_{0}^{2}\in(0, \infty).
\end{equation}
Since it is clear that $\int_{\R^{N}} V_{\e}(x)|w_{\e}|^{2}dx\rightarrow \int_{\R^{N}} V_{0} |w|^{2}dx$, it remains to show that
\begin{equation}\label{limwr*}
\lim_{\e\rightarrow 0}[w_{\e}]^{2}_{A_{\e}}=[w]^{2}.
\end{equation}
Using Lemma $5$ in \cite{PP}, we know that 
\begin{equation}\label{PPlem}
[\eta_{\e} w]\rightarrow [w] \mbox{ as } \e\rightarrow 0.
\end{equation}
On the other hand
\begin{align*}
[w_{\e}]_{A_{\e}}^{2}
&=\iint_{\R^{2N}} \frac{|e^{\imath A(0)\cdot x}\eta_{\e}(x)w(x)-e^{\imath A_{\e}(\frac{x+y}{2})\cdot (x-y)}e^{\imath A(0)\cdot y} \eta_{\e}(y)w(y)|^{2}}{|x-y|^{N+2s}} dx dy \nonumber \\
&=[\eta_{\e} w]^{2}
+\iint_{\R^{2N}} \frac{\eta_{\e}^2(y)w^2(y) |e^{\imath [A_{\e}(\frac{x+y}{2})-A(0)]\cdot (x-y)}-1|^{2}}{|x-y|^{N+2s}} dx dy\\
&\quad+2\Re \iint_{\R^{2N}} \frac{(\eta_{\e}(x)w(x)-\eta_{\e}(y)w(y))\eta_{\e}(y)w(y)(1-e^{-\imath [A_{\e}(\frac{x+y}{2})-A(0)]\cdot (x-y)})}{|x-y|^{N+2s}} dx dy \\
&=: [\eta_{\e} w]^{2}+X_{\e}+2Y_{\e}.
\end{align*}
Then, in view of 
$|Y_{\e}|\leq [\eta_{\e} w] \sqrt{X_{\e}}$ and \eqref{PPlem}, it is suffices to prove that $X_{\e}\rightarrow 0$ as $\e\rightarrow 0$ to deduce that \eqref{limwr*} holds.
Let us note that for $0<\beta<\alpha/({1+\alpha-s})$, 
\begin{equation}\label{Ye}
\begin{split}
X_{\e}
&\leq \int_{\R^{N}} w^{2}(y) dy \int_{|x-y|\geq\e^{-\beta}} \frac{|e^{\imath [A_{\e}(\frac{x+y}{2})-A(0)]\cdot (x-y)}-1|^{2}}{|x-y|^{N+2s}} dx\\
&\quad +\int_{\R^{N}} w^{2}(y) dy  \int_{|x-y|<\e^{-\beta}} \frac{|e^{\imath [A_{\e}(\frac{x+y}{2})-A(0)]\cdot (x-y)}-1|^{2}}{|x-y|^{N+2s}} dx\\
&=:X^{1}_{\e}+X^{2}_{\e}.
\end{split}
\end{equation}
Using $|e^{\imath t}-1|^{2}\leq 4$ and $w\in H^{s}(\R^{N}, \R)$, we get
\begin{equation}\label{Ye1}
X_{\e}^{1}\leq C \int_{\R^{N}} w^{2}(y) dy \int_{\e^{-\beta}}^\infty \rho^{-1-2s} d\rho\leq C \e^{2\beta s} \rightarrow 0.
\end{equation}
Since $|e^{\imath t}-1|^{2}\leq t^{2}$ for all $t\in \R$, $A\in C^{0,\alpha}(\R^N,\R^N)$ with $\alpha\in(0,1]$, and $|x+y|^{2}\leq 2(|x-y|^{2}+4|y|^{2})$, we have
\begin{equation}\label{Ye2}
\begin{split}
X^{2}_{\e}&
	\leq \int_{\R^{N}} w^{2}(y) dy  \int_{|x-y|<\e^{-\beta}} \frac{|A_{\e}\left(\frac{x+y}{2}\right)-A(0)|^{2} }{|x-y|^{N+2s-2}} dx \\
	&\leq C\e^{2\alpha} \int_{\R^{N}} w^{2}(y) dy  \int_{|x-y|<\e^{-\beta}} \frac{|x+y|^{2\alpha} }{|x-y|^{N+2s-2}} dx \\
	&\leq C\e^{2\alpha} \left(\int_{\R^{N}} w^{2}(y) dy  \int_{|x-y|<\e^{-\beta}} \frac{1 }{|x-y|^{N+2s-2-2\alpha}} dx\right.\\
	&\qquad \qquad+ \left. \int_{\R^{N}} |y|^{2\alpha} w^{2}(y) dy  \int_{|x-y|<\e^{-\beta}} \frac{1}{|x-y|^{N+2s-2}} dx\right) \\
	&=: C\e^{2\alpha} (X^{2, 1}_{\e}+X^{2, 2}_{\e}).
	\end{split}
	\end{equation}	
	Then
	\begin{equation}\label{Ye21}
	X^{2, 1}_{\e}
	= C  \int_{\R^{N}} w^{2}(y) dy \int_0^{\e^{-\beta}} \rho^{1+2\alpha-2s} d\rho
	\leq C\e^{-2\beta(1+\alpha-s)}.
	\end{equation}
	On the other hand, using \eqref{remdecay}, we infer that
	\begin{equation}\label{Ye22}
	\begin{split}
	 X^{2, 2}_{\e}
	 &\leq C  \int_{\R^{N}} |y|^{2\alpha} w^{2}(y) dy \int_0^{\e^{-\beta}}\rho^{1-2s} d\rho  \\
	&\leq C \e^{-2\beta(1-s)} \left[\int_{B_1}  w^{2}(y) dy + \int_{\R^{N}\setminus B_{1}} \frac{1}{|y|^{2(N+2s)-2\alpha}} dy \right]  \\
	&\leq C \e^{-2\beta(1-s)}.
	\end{split}
	\end{equation}
	Taking into account \eqref{Ye}, \eqref{Ye1}, \eqref{Ye2}, \eqref{Ye21} and \eqref{Ye22} we can conclude that $X_{\e}\rightarrow 0$. Therefore \eqref{limwr} holds.
Now, let $t_{\e}>0$ be the unique number such that 
\begin{equation*}
J_{\e}(t_{\e} w_{\e})=\max_{t\geq 0} J_{\e}(t w_{\e}).
\end{equation*}
Then $t_{\e}$ verifies 
\begin{equation}\label{AS1}
\|w_{\e}\|_{\e}^{2}=\int_{\R^{N}} g_{\e}(x, t_{\e}^{2} |w_{\e}|^{2}) |w_{\e}|^{2}dx=\int_{\R^{N}} f(t_{\e}^{2} |w_{\e}|^{2}) |w_{\e}|^{2}dx
\end{equation}
where we used $\supp(\eta)\subset \Lambda$ and $g=f$ on $\Lambda$.
Let us prove that $t_{\e}\rightarrow 1$ as $\e\rightarrow 0$. Using $\eta=1$ in $B_{\frac{\delta}{2}}$ and that $w$ is a continuous positive function, we can see that $(f_4)$ yields
$$
\|w_{\e}\|_{\e}^{2}\geq f(t_{\e}^{2}\alpha^{2}_{0})\int_{B_{\frac{\delta}{2}}}|w|^{2}dx, 
$$
where $\alpha_{0}=\min_{\bar{B}_{\delta/2}} w>0$. So, if $t_{\e}\rightarrow \infty$ as $\e\rightarrow 0$, then we can use $(f_3)$ and \eqref{limwr} to deduce that $\|w\|_{0}^{2}= \infty$, which gives a contradiction.
On the other hand, if $t_{\e}\rightarrow 0$ as $\e\rightarrow 0$, we can use the growth assumptions on $f$ and \eqref{limwr} to infer that $\|w\|_{0}^{2}= 0$ which is impossible.
In conclusion, $t_{\e}\rightarrow t_{0}\in (0, \infty)$ as $\e\rightarrow 0$.
Now, taking the limit as $\e\rightarrow 0$ in \eqref{AS1} and using \eqref{limwr}, we can see that 
\begin{equation}\label{AS2}
\|w\|_{0}^{2}=\int_{\R^{N}} f(t_{0}^{2} |w|^{2}) |w|^{2}dx.
\end{equation}
From $w\in \mathcal{N}_{0}$ and $(f_4)$, it follows that $t_{0}=1$. Then, using \eqref{limwr}, $t_{\e}\rightarrow 1$ and applying the Dominated Convergence Theorem, we obtain that $\lim_{\e\rightarrow 0} J_{\e}(t_{\e} w_{\e})=I_{0}(w)=c_{0}$.
Since $c_{\e}\leq \max_{t\geq 0} J_{\e}(t w_{\e})=J_{\e}(t_{\e} w_{\e})$, we can conclude  that
$\limsup_{\e\rightarrow 0} c_{\e}\leq c_{0}$.
\end{proof}

\noindent
Let us recall the following result for the autonomous problem \eqref{APe} (see \cite{A1, FigS}).
\begin{lem}\label{FS}
Let $(u_{n})\subset \mathcal{N}_{0}$ be a sequence satisfying $I_{0}(u_{n})\rightarrow c_{0}$. Then, up to subsequences, one of the following alternatives holds:
\begin{compactenum}[(i)]
\item $(u_{n})$ strongly converges in $H^{s}(\R^{N}, \R)$, 
\item there exists a sequence $(\tilde{y}_{n})\subset \R^{N}$ such that,  up to a subsequence, $v_{n}(x)=u_{n}(x+\tilde{y}_{n})$ strongly converges in $H^{s}(\R^{N}, \R)$.
\end{compactenum}
\end{lem}


\noindent
Now we prove the following useful compactness result.
\begin{lem}\label{prop3.3}
Let $\e_{n}\rightarrow 0$ and $u_{\e_{n}}\in H^{s}_{\e_{n}}$ be such that $J_{\e_{n}}(u_{\e_{n}})=c_{\e_{n}}$ and $J'_{\e_{n}}(u_{\e_{n}})=0$. Then there exists $(\tilde{y}_{\e_{n}})\subset \R^{N}$ such that $v_{n}(x)=|u_{\e_{n}}|(x+\tilde{y}_{\e_{n}})$ has a convergent subsequence in $H^{s}(\R^{N}, \R)$. Moreover, up to a subsequence, $y_{n}=\e_{n} \tilde{y}_{\e_{n}}\rightarrow y_{0}$ for some $y_{0}\in \Lambda$ such that $V(y_{0})=V_{0}$.
\end{lem}
\begin{proof}
Hereafter, we write $(\tilde{y}_{n})$ and $(u_{n})$ to denote the sequences $(\tilde{y}_{\e_{n}})$ and $(u_{\e_{n}})$, respectively.
Taking into account $\langle J'_{\e_{n}}(u_{n}), u_{n}\rangle=0$, $J_{\e_{n}}(u_{n})= c_{\e_{n}}$ and Lemma \ref{AMlem1} it is easy to see that $(u_{n})$ is bounded in $H^{s}_{\e_{n}}$. 
Then, there exists $C>0$ (independent of $n$) such that $\|u_{n}\|_{\e_{n}}\leq C$ for all $n\in \mathbb{N}$. Moreover, from Lemma \ref{DI}, we also know that $(|u_{n}|)$ is bounded in $H^{s}(\R^{N}, \R)$.\\
Now we prove that there exist a sequence $(\tilde{y}_{n})\subset \R^{N}$ and constants $R, \gamma>0$ such that
\begin{equation}\label{sacchi}
\liminf_{n\rightarrow \infty}\int_{B_{R}(\tilde{y}_{n})} |u_{n}|^{2} \, dx\geq \gamma>0.
\end{equation}
If by contradiction \eqref{sacchi} does not hold, then for all $R>0$ we get
$$
\lim_{n\rightarrow \infty}\sup_{y\in \R^{N}}\int_{B_{R}(y)} |u_{n}|^{2} \, dx=0.
$$
From the boundedness of $(|u_{n}|)$ and Lemma $2.2$ in \cite{FQT} we can see that $|u_{n}|\rightarrow 0$ in $L^{q}(\R^{N}, \R)$ for any $q\in (2, 2^{*}_{s})$. 
This fact together with $|g_{\e_{n}}(x, t)|\leq \delta+C_{\delta}|t|^{\frac{q-2}{2}}$ in $\R^{N}\times \R$ and the boundedness of $(|u_{n}|)$ in $L^{2}(\R^{N}, \R)$ yields that
\begin{align}\label{glimiti}
\lim_{n\rightarrow \infty}\int_{\R^{N}} g_{\e_{n}} (x, |u_{n}|^{2}) |u_{n}|^{2} \,dx=0= \lim_{n\rightarrow \infty}\int_{\R^{N}} G_{\e_{n}}(x, |u_{n}|^{2}) \, dx.
\end{align}
Taking into account $\langle J'_{\e_{n}}(u_{n}), u_{n}\rangle=0$ and \eqref{glimiti}, we can  infer that $\|u_{n}\|_{\e_{n}}\rightarrow 0$ as $n\rightarrow \infty$. This is impossible because $u_{n}\in \N_{\e_{n}}$, and using $(g_1)$ and $(g_2)$ we can find $\alpha_{0}>0$ such that $\|u_{n}\|^{2}_{\e_{n}}\geq \alpha_{0}$ for all $n\in \mathbb{N}$.
Now, we set $v_{n}(x)=|u_{n}|(x+\tilde{y}_{n})$. Then $(v_{n})$ is bounded in $H^{s}(\R^{N}, \R)$, and we may assume that 
$v_{n}\rightharpoonup v\not\equiv 0$ in $H^{s}(\R^{N}, \R)$  as $n\rightarrow \infty$.
Fix $t_{n}>0$ such that $\tilde{v}_{n}=t_{n} v_{n}\in \mathcal{N}_{0}$. By Lemma \ref{DI} and $u_{n}\in \mathcal{N}_{\e_{n}}$ we can see that 
$$
c_{0}\leq I_{0}(\tilde{v}_{n})\leq \max_{t\geq 0}J_{\e_{n}}(tu_{n})= J_{\e_{n}}(u_{n})
$$
which together with Lemma \ref{AMlem1} implies that $I_{0}(\tilde{v}_{n})\rightarrow c_{0}$. In particular, $\tilde{v}_{n}\nrightarrow 0$ in $H^{s}(\R^{N}, \R)$.
Since $(v_{n})$ and $(\tilde{v}_{n})$ are bounded in $H^{s}(\R^{N}, \R)$ and $\tilde{v}_{n}\nrightarrow 0$  in $H^{s}(\R^{N}, \R)$, we deduce that $t_{n}\rightarrow t^{*}\geq 0$. Indeed $t^{*}>0$ due to $\tilde{v}_{n}\nrightarrow 0$  in $H^{s}(\R^{N}, \R)$. From the uniqueness of the weak limit, we can deduce that $\tilde{v}_{n}\rightharpoonup \tilde{v}=t^{*}v\not\equiv 0$ in $H^{s}(\R^{N}, \R)$. 
This fact combined with Lemma \ref{FS} yields
\begin{equation}\label{elena}
\tilde{v}_{n}\rightarrow \tilde{v} \mbox{ in } H^{s}(\R^{N}, \R).
\end{equation} 
As a consequence, $v_{n}\rightarrow v$ in $H^{s}(\R^{N}, \R)$ as $n\rightarrow \infty$.

Now, we put $y_{n}=\e_{n}\tilde{y}_{n}$ and we show that $(y_{n})$ admits a subsequence, still denoted by $y_{n}$, such that $y_{n}\rightarrow y_{0}$ for some $y_{0}\in \Lambda$ such that $V(y_{0})=V_{0}$. Firstly, we prove that $(y_{n})$ is bounded. Assume by contradiction that, up to a subsequence, $|y_{n}|\rightarrow \infty$ as $n\rightarrow \infty$. Take $R>0$ such that $\Lambda \subset B_{R}$. Since we may suppose that  $|y_{n}|>2R$, we have that for any $z\in B_{R/\e_{n}}$ 
$$
|\e_{n}z+y_{n}|\geq |y_{n}|-|\e_{n}z|>R.
$$
Now, using $(u_{n})\subset \N_{\e_{n}}$, $(V_{1})$, Lemma \ref{DI} and the change of variable $x\mapsto z+\tilde{y}_{n}$ we obtain that 
\begin{align}\label{pasq}
[v_{n}]^{2}+\int_{\R^{N}} V_{0} v_{n}^{2}\, dx &\leq \int_{\R^{N}} g(\e_{n} x+y_{n}, |v_{n}|^{2}) |v_{n}|^{2} \, dx \nonumber\\
&\leq \int_{B_{\frac{R}{\e_{n}}}} \tilde{f}(|v_{n}|^{2}) |v_{n}|^{2} \, dx+\int_{\R^{N}\setminus B_{\frac{R}{\e_{n}}}} f(|v_{n}|^{2}) |v_{n}|^{2} \, dx.
\end{align}
Then, recalling that $v_{n}\rightarrow v$ in $H^{s}(\R^{N}, \R)$ as $n\rightarrow \infty$ and $\tilde{f}(t)\leq \frac{V_{0}}{k}$, we can see that (\ref{pasq}) yields
$$
\min\left\{1, V_{0}\left(1-\frac{1}{k}\right) \right\} \left([v_{n}]^{2}+\int_{\R^{N}} |v_{n}|^{2}\, dx\right)=o_{n}(1),
$$
that is $v_{n}\rightarrow 0$ in $H^{s}(\R^{N}, \R)$, which gives a contradiction. Therefore, $(y_{n})$ is bounded and we may assume that $y_{n}\rightarrow y_{0}\in \R^{N}$. If $y_{0}\notin \overline{\Lambda}$, then we can argue as before to infer that $v_{n}\rightarrow 0$ in $H^{s}(\R^{N}, \R)$, which is impossible. Hence $y_{0}\in \overline{\Lambda}$. Let us note that if $V(y_{0})=V_{0}$, then we can infer that $y_{0}\notin \partial \Lambda$ in view of $(V_2)$. Therefore, it is enough to verify that $V(y_{0})=V_{0}$. Suppose by contradiction that $V(y_{0})>V_{0}$.
Then, using (\ref{elena}), Fatou's Lemma, the invariance of  $\R^{N}$ by translations, Lemma \ref{DI} and Lemma \ref{AMlem1}, we get 
\begin{align*}
c_{0}=I_{0}(\tilde{v})&<\frac{1}{2}[\tilde{v}]^{2}+\frac{1}{2}\int_{\R^{N}} V(y_{0})\tilde{v}^{2} \, dx-\frac{1}{2}\int_{\R^{N}} F(|\tilde{v}|^{2})\, dx\\
&\leq \liminf_{n\rightarrow \infty}\left[\frac{1}{2}[\tilde{v}_{n}]^{2}+\frac{1}{2}\int_{\R^{N}} V(\e_{n}x+y_{n}) |\tilde{v}_{n}|^{2} \, dx-\frac{1}{2}\int_{\R^{N}} F(|\tilde{v}_{n}|^{2})\, dx  \right] \\
&\leq \liminf_{n\rightarrow \infty}\left[\frac{t_{n}^{2}}{2}[|u_{n}|]^{2}+\frac{t_{n}^{2}}{2}\int_{\R^{N}} V(\e_{n}z) |u_{n}|^{2} \, dz-\frac{1}{2}\int_{\R^{N}} F(|t_{n} u_{n}|^{2})\, dz  \right] \\
&\leq \liminf_{n\rightarrow \infty} J_{\e_{n}}(t_{n} u_{n}) \leq \liminf_{n\rightarrow \infty} J_{\e_{n}}(u_{n})\leq c_{0}
\end{align*}
which gives a contradiction. This ends the proof of lemma.
\end{proof}

\noindent
Now we prove the following key lemma which will be fundamental to establish that the solutions of \eqref{Pe} are indeed solutions  of \eqref{P}.
\begin{lem}\label{moser} 
Let $\e_{n}\rightarrow 0$ and $u_{n}\in H^{s}_{\e_{n}}$ be such that $J_{\e_{n}}(u_{n})=c_{\e_{n}}$ and $J'_{\e_{n}}(u_{n})=0$.
Then $v_{n}=|u_{n}|(\cdot+\tilde{y}_{n})$ satisfies $v_{n}\in L^{\infty}(\R^{N},\R)$ and there exists $C>0$ such that 
$$
\|v_{n}\|_{L^{\infty}(\R^{N})}\leq C \mbox{ for all } n\in \mathbb{N},
$$
where $(\tilde{y}_{n})$ is given by Lemma \ref{prop3.3}.
Moreover
$$
\lim_{|x|\rightarrow \infty} v_{n}(x)=0 \mbox{ uniformly in } n\in \mathbb{N}.
$$
\end{lem}
\begin{proof}
For any $L>0$ we define $u_{L,n}:=\min\{|u_{n}|, L\}\geq 0$ and we set $v_{L, n}=u_{L,n}^{2(\beta-1)}u_{n}$, where $\beta>1$ will be chosen later.
Taking $v_{L, n}$ as test function in (\ref{Pe}) we can see that
\begin{align}\label{conto1FF}
&\Re\left(\iint_{\R^{2N}} \frac{(u_{n}(x)-u_{n}(y)e^{\imath A_{\e_{n}}(\frac{x+y}{2})\cdot (x-y)})}{|x-y|^{N+2s}} \overline{(u_{n}(x)u_{L,n}^{2(\beta-1)}(x)-u_{n}(y)u_{L,n}^{2(\beta-1)}(y)e^{\imath A_{\e_{n}}(\frac{x+y}{2})\cdot (x-y)})} \, dx dy\right)   \nonumber \\
&=\int_{\R^{N}} g_{\e_{n}}(x, |u_{n}|^{2}) |u_{n}|^{2}u_{L,n}^{2(\beta-1)}  \,dx-\int_{\R^{N}} V_{\e_{n}}(x) |u_{n}|^{2} u_{L,n}^{2(\beta-1)} \, dx.
\end{align}
Let us note that
\begin{align*}
&\Re\left[(u_{n}(x)-u_{n}(y)e^{\imath A_{\e_{n}}(\frac{x+y}{2})\cdot (x-y)})\overline{(u_{n}(x)u_{L,n}^{2(\beta-1)}(x)-u_{n}(y)u_{L,n}^{2(\beta-1)}(y)e^{\imath A_{\e_{n}}(\frac{x+y}{2})\cdot (x-y)})}\right] \\
&=\Re\Bigl[|u_{n}(x)|^{2}u_{L,n}^{2(\beta-1)}(x)-u_{n}(x)\overline{u_{n}(y)} u_{L,n}^{2(\beta-1)}(y)e^{-\imath A_{\e_{n}}(\frac{x+y}{2})\cdot (x-y)}-u_{n}(y)\overline{u_{n}(x)} u_{L,n}^{2(\beta-1)}(x) e^{\imath A_{\e_{n}}(\frac{x+y}{2})\cdot (x-y)} \\
&\qquad +|u_{n}(y)|^{2}u_{L,n}^{2(\beta-1)}(y) \Bigr] \\
&\geq (|u_{n}(x)|^{2}u_{L,n}^{2(\beta-1)}(x)-|u_{n}(x)||u_{n}(y)|u_{L,n}^{2(\beta-1)}(y)-|u_{n}(y)||u_{n}(x)|u_{L,n}^{2(\beta-1)}(x)+|u_{n}(y)|^{2}u^{2(\beta-1)}_{L,n}(y) \\
&=(|u_{n}(x)|-|u_{n}(y)|)(|u_{n}(x)|u_{L,n}^{2(\beta-1)}(x)-|u_{n}(y)|u_{L,n}^{2(\beta-1)}(y)),
\end{align*}
so we have
\begin{align}\label{realeF}
&\Re\left(\iint_{\R^{2N}} \frac{(u_{n}(x)-u_{n}(y)e^{\imath A_{\e_{n}}(\frac{x+y}{2})\cdot (x-y)})}{|x-y|^{N+2s}} \overline{(u_{n}(x)u_{L,n}^{2(\beta-1)}(x)-u_{n}(y)u_{L,n}^{2(\beta-1)}(y)e^{\imath A_{\e_{n}}(\frac{x+y}{2})\cdot (x-y)})} \, dx dy\right) \nonumber\\
&\geq \iint_{\R^{2N}} \frac{(|u_{n}(x)|-|u_{n}(y)|)}{|x-y|^{N+2s}} (|u_{n}(x)|u_{L,n}^{2(\beta-1)}(x)-|u_{n}(y)|u_{L,n}^{2(\beta-1)}(y))\, dx dy.
\end{align}
For all $t\geq 0$, let us define
\begin{equation*}
\gamma(t)=\gamma_{L, \beta}(t)=t t_{L}^{2(\beta-1)}
\end{equation*}
where  $t_{L}=\min\{t, L\}$. 
Let us observe that, since $\gamma$ is an increasing function, it holds
\begin{align*}
(a-b)(\gamma(a)- \gamma(b))\geq 0 \quad \mbox{ for any } a, b\in \R.
\end{align*}
Let us consider the functions 
\begin{equation*}
\Lambda(t)=\frac{|t|^{2}}{2} \quad \mbox{ and } \quad \Gamma(t)=\int_{0}^{t} (\gamma'(\tau))^{\frac{1}{2}} d\tau,
\end{equation*}
and we note that
\begin{equation}\label{Gg}
\Lambda'(a-b)(\gamma(a)-\gamma(b))\geq |\Gamma(a)-\Gamma(b)|^{2} \mbox{ for any } a, b\in\R. 
\end{equation}
Indeed, for any $a, b\in \R$ such that $a<b$, the Jensen inequality yields
\begin{align*}
\Lambda'(a-b)(\gamma(a)-\gamma(b))=(a-b)\int_{b}^{a} \gamma'(t)dt=(a-b)\int_{b}^{a} (\Gamma'(t))^{2}dt\geq \left(\int_{b}^{a} \Gamma'(t) dt\right)^{2}=(\Gamma(a)-\Gamma(b))^{2}.
\end{align*}
By \eqref{Gg}, it follows that
\begin{align}\label{Gg1}
|\Gamma(|u_{n}(x)|)- \Gamma(|u_{n}(y)|)|^{2} \leq (|u_{n}(x)|- |u_{n}(y)|)((|u_{n}|u_{L,n}^{2(\beta-1)})(x)- (|u_{n}|u_{L,n}^{2(\beta-1)})(y)). 
\end{align}
Then, in view of \eqref{realeF} and \eqref{Gg1}, we obtain
\begin{align}\label{conto1FFF}
&\Re\left(\iint_{\R^{2N}} \frac{(u_{n}(x)-u_{n}(y)e^{\imath A_{\e_{n}}(\frac{x+y}{2})\cdot (x-y)})}{|x-y|^{N+2s}} \overline{(u_{n}(x)u_{L,n}^{2(\beta-1)}(x)-u_{n}(y)u_{L,n}^{2(\beta-1)}(y)e^{\imath A_{\e_{n}}(\frac{x+y}{2})\cdot (x-y)})} \, dx dy\right) \nonumber \\
&\qquad \geq [\Gamma(|u_{n}|)]^{2}.
\end{align}
Since $\Gamma(|u_{n}|)\geq \frac{1}{\beta} |u_{n}| u_{L,n}^{\beta-1}$ and using the fractional Sobolev embedding $\mathcal{D}^{s,2}(\R^{N}, \R)\subset L^{\2}(\R^{N}, \R)$ (see \cite{DPV}), we deduce that 
\begin{equation}\label{SS1}
[\Gamma(|u_{n}|)]^{2}\geq S_{*} \|\Gamma(|u_{n}|)\|^{2}_{L^{\2}(\R^{N})}\geq \left(\frac{1}{\beta}\right)^{2} S_{*}\||u_{n}| u_{L,n}^{\beta-1}\|^{2}_{L^{\2}(\R^{N})}.
\end{equation}
Putting together \eqref{conto1FF}, \eqref{conto1FFF} and \eqref{SS1} we can infer that
\begin{align}\label{BMS}
\left(\frac{1}{\beta}\right)^{2} S_{*}\||u_{n}| u_{L,n}^{\beta-1}\|^{2}_{L^{\2}(\R^{N})}+\int_{\R^{N}} V_{\e_{n}}(x)|u_{n}|^{2}u_{L,n}^{2(\beta-1)} dx\leq \int_{\R^{N}} g_{\e_{n}}(x, |u_{n}|^{2}) |u_{n}|^{2} u_{L,n}^{2(\beta-1)} dx.
\end{align}
On the other hand, from assumptions $(g_1)$ and $(g_2)$, for any $\xi>0$ there exists $C_{\xi}>0$ such that
\begin{equation}\label{SS2}
g_{\e}(x, t^{2})t^{2}\leq \xi |t|^{2}+C_{\xi}|t|^{\2} \mbox{ for all } t\in \R.
\end{equation}
Taking $\xi\in (0, V_{0})$ and using \eqref{BMS} and \eqref{SS2} we can see that
\begin{equation}\label{simo1}
\|w_{L,n}\|_{L^{\2}(\R^{N})}^{2}\leq C \beta^{2} \int_{\R^{N}} |u_{n}|^{\2}u_{L,n}^{2(\beta-1)} \, dx,
\end{equation}
where $w_{L,n}:=|u_{n}| u_{L,n}^{\beta-1}$. \\
Now, we take $\beta=\frac{\2}{2}$ and fix $R>0$. Recalling that $0\leq u_{L,n}\leq |u_{n}|$ and applying the H\"older inequality we have
\begin{align}\label{simo2}
\int_{\R^{N}} |u_{n}|^{\2}u_{L,n}^{2(\beta-1)}dx&=\int_{\R^{N}} |u_{n}|^{\2-2} |u_{n}|^{2} u_{L,n}^{\2-2}dx \nonumber\\
&=\int_{\R^{N}} |u_{n}|^{\2-2} (|u_{n}| u_{L,n}^{\frac{\2-2}{2}})^{2}dx \nonumber\\
&\leq \int_{\{|u_{n}|<R\}} R^{\2-2} |u_{n}|^{\2} dx+\int_{\{|u_{n}|>R\}} |u_{n}|^{\2-2} (|u_{n}| u_{L,n}^{\frac{\2-2}{2}})^{2}dx \nonumber\\
&\leq \int_{\{|u_{n}|<R\}} R^{\2-2} |u_{n}|^{\2} dx+\left(\int_{\{|u_{n}|>R\}} |u_{n}|^{\2} dx\right)^{\frac{\2-2}{\2}} \left(\int_{\R^{N}} (|u_{n}| u_{L,n}^{\frac{\2-2}{2}})^{\2}dx\right)^{\frac{2}{\2}}.
\end{align}
Since $(|u_{n}|)$ is bounded in $H^{s}(\R^{N}, \R)$, we can see that for any $R$ sufficiently large
\begin{equation}\label{simo3}
\left(\int_{\{|u_{n}|>R\}} |u_{n}|^{\2} dx\right)^{\frac{\2-2}{\2}}\leq  \frac{1}{2C\beta^{2}}.
\end{equation}
Putting together \eqref{simo1}, \eqref{simo2} and \eqref{simo3} we get
\begin{equation*}
\left(\int_{\R^{N}} (|u_{n}| u_{L,n}^{\frac{\2-2}{2}})^{\2} \, dx\right)^{\frac{2}{\2}}\leq C\beta^{2} \int_{\R^{N}} R^{\2-2} |u_{n}|^{\2} dx<\infty
\end{equation*}
and taking the limit as $L\rightarrow \infty$ we obtain $|u_{n}|\in L^{\frac{(\2)^{2}}{2}}(\R^{N},\R)$.
Now, using $0\leq u_{L,n}\leq |u_{n}|$ and by passing to the limit as $L\rightarrow \infty$ in \eqref{simo1}, we have
\begin{equation*}
\||u_{n}|\|_{L^{\beta\2}(\R^{N})}^{2\beta}\leq C \beta^{2} \int_{\R^{N}} |u_{n}|^{\2+2(\beta-1)} \, dx,
\end{equation*}
from which we deduce that
\begin{equation*}
\left(\int_{\R^{N}} |u_{n}|^{\beta\2} dx\right)^{\frac{1}{(\beta-1)\2}}\leq (C \beta)^{\frac{1}{\beta-1}} \left(\int_{\R^{N}} |u_{n}|^{\2+2(\beta-1)}\, dx\right)^{\frac{1}{2(\beta-1)}}.
\end{equation*}
For $m\geq 1$ we define $\beta_{m+1}$ inductively so that $\2+2(\beta_{m+1}-1)=\2 \beta_{m}$ and $\beta_{1}=\frac{\2}{2}$. Then we have
\begin{equation*}
\left(\int_{\R^{N}} |u_{n}|^{\beta_{m+1}\2} dx\right)^{\frac{1}{(\beta_{m+1}-1)\2}}\leq (C \beta_{m+1})^{\frac{1}{\beta_{m+1}-1}} \left(\int_{\R^{N}} |u_{n}|^{\2\beta_{m}}\, dx \right)^{\frac{1}{\2(\beta_{m}-1)}}.
\end{equation*}
Let us define
$$
D_{m}=\left(\int_{\R^{N}} |u_{n}|^{\2\beta_{m}}\, dx \right)^{\frac{1}{\2(\beta_{m}-1)}}.
$$
Using an iteration argument, we can find $C_{0}>0$ independent of $m$ such that 
$$
D_{m+1}\leq \prod_{k=1}^{m} (C \beta_{k+1})^{\frac{1}{\beta_{k+1}-1}}  D_{1}\leq C_{0} D_{1}.
$$
Taking the limit as $m\rightarrow \infty$ we get 
\begin{equation}\label{UBu}
\||u_{n}|\|_{L^{\infty}(\R^{N})}\leq C_{0}D_{1}=:K \mbox{ for all } n\in \mathbb{N}.
\end{equation}
Moreover, by interpolation, $(|u_{n}|)$ strongly converges in $L^{r}(\R^{N}, \R)$ for all $r\in (2, \infty)$, and in view of the growth assumptions on $g$, also $g_{\e_{n}}(x, |u_{n}|^{2})|u_{n}|$ strongly converges  in the same Lebesgue spaces. \\
Now we aim to prove that $|u_{n}|$ is a weak subsolution to 
\begin{equation}\label{Kato0}
\left\{
\begin{array}{ll}
(-\Delta)^{s}v+V_{\e_{n}}(x) v=g_{\e_{n}}(x, v^{2})v &\mbox{ in } \R^{N} \\
v\geq 0 \quad \mbox{ in } \R^{N}.
\end{array}
\right.
\end{equation}
Fix $\varphi\in C^{\infty}_{c}(\R^{N}, \R)$ such that $\varphi\geq 0$, and we take $\psi_{\delta, n}=\frac{u_{n}}{u_{\delta, n}}\varphi$ as test function in \eqref{Pe}, where we set $u_{\delta,n}=\sqrt{|u_{n}|^{2}+\delta^{2}}$ for $\delta>0$. We note that $\psi_{\delta, n}\in H^{s}_{\e_{n}}$ for all $\delta>0$ and $n\in \mathbb{N}$. Indeed, $\int_{\R^{N}} V_{\e_{n}}(x) |\psi_{\delta,n}|^{2} dx\leq \int_{\supp(\varphi)} V_{\e_{n}} (x)\varphi^{2} dx<\infty$. 
On the other hand, we can observe
\begin{align*}
\psi_{\delta,n}(x)-\psi_{\delta,n}(y)e^{\imath A_{\e_{n}}(\frac{x+y}{2})\cdot (x-y)}&=\left(\frac{u_{n}(x)}{u_{\delta,n}(x)}\right)\varphi(x)-\left(\frac{u_{n}(y)}{u_{\delta,n}(y)}\right)\varphi(y)e^{\imath A_{\e_{n}}(\frac{x+y}{2})\cdot (x-y)}\\
&=\left[\left(\frac{u_{n}(x)}{u_{\delta,n}(x)}\right)-\left(\frac{u_{n}(y)}{u_{\delta,n}(x)}\right)e^{\imath A_{\e_{n}}(\frac{x+y}{2})\cdot (x-y)}\right]\varphi(x) \\
&\quad +\left[\varphi(x)-\varphi(y)\right] \left(\frac{u_{n}(y)}{u_{\delta,n}(x)}\right) e^{\imath A_{\e_{n}}(\frac{x+y}{2})\cdot (x-y)} \\
&\quad +\left(\frac{u_{n}(y)}{u_{\delta,n}(x)}-\frac{u_{n}(y)}{u_{\delta,n}(y)}\right)\varphi(y) e^{\imath A_{\e_{n}}(\frac{x+y}{2})\cdot (x-y)}. 
\end{align*}
Using $|z+w+k|^{2}\leq 4(|z|^{2}+|w|^{2}+|k|^{2})$ for all $z,w,k\in \C$, $|e^{\imath t}|=1$ for all $t\in \R$, $u_{\delta,n}\geq \delta$, $|\frac{u_{n}}{u_{\delta,n}}|\leq 1$, \eqref{UBu} and $|\sqrt{|z|^{2}+\delta^{2}}-\sqrt{|w|^{2}+\delta^{2}}|\leq ||z|-|w||$ for all $z, w\in \C$, we obtain that
\begin{align*}
&|\psi_{\delta,n}(x)-\psi_{\delta,n}(y)e^{\imath A_{\e_{n}}(\frac{x+y}{2})\cdot (x-y)}|^{2} \\
&\leq \frac{4}{\delta^{2}}|u_{n}(x)-u_{n}(y)e^{\imath A_{\e_{n}}(\frac{x+y}{2})\cdot (x-y)}|^{2}\|\varphi\|^{2}_{L^{\infty}(\R^{N})} +\frac{4}{\delta^{2}}|\varphi(x)-\varphi(y)|^{2} \||u_{n}|\|^{2}_{L^{\infty}(\R^{N})} \\
&\quad+\frac{4}{\delta^{4}} \||u_{n}|\|^{2}_{L^{\infty}(\R^{N})} \|\varphi\|^{2}_{L^{\infty}(\R^{N})} |u_{\delta,n}(y)-u_{\delta,n}(x)|^{2} \\
&\leq \frac{4}{\delta^{2}}|u_{n}(x)-u_{n}(y)e^{\imath A_{\e_{n}}(\frac{x+y}{2})\cdot (x-y)}|^{2}\|\varphi\|^{2}_{L^{\infty}(\R^{N})} +\frac{4K^{2}}{\delta^{2}}|\varphi(x)-\varphi(y)|^{2} \\
&\quad+\frac{4K^{2}}{\delta^{4}} \|\varphi\|^{2}_{L^{\infty}(\R^{N})} ||u_{n}(y)|-|u_{n}(x)||^{2}. 
\end{align*}
Since $u_{n}\in H^{s}_{\e_{n}}$, $|u_{n}|\in H^{s}(\R^{N}, \R)$ (by Lemma \ref{DI}) and $\varphi\in C^{\infty}_{c}(\R^{N}, \R)$, we deduce that $\psi_{\delta,n}\in H^{s}_{\e_{n}}$. 
Then we have
\begin{align}\label{Kato1}
&\Re\left[\iint_{\R^{2N}} \frac{(u_{n}(x)-u_{n}(y)e^{\imath A_{\e_{n}}(\frac{x+y}{2})\cdot (x-y)})}{|x-y|^{N+2s}} \left(\frac{\overline{u_{n}(x)}}{u_{\delta,n}(x)}\varphi(x)-\frac{\overline{u_{n}(y)}}{u_{\delta,n}(y)}\varphi(y)e^{-\imath A_{\e_{n}}(\frac{x+y}{2})\cdot (x-y)}  \right) dx dy\right] \nonumber\\
&+\int_{\R^{N}} V_{\e_{n}}(x)\frac{|u_{n}|^{2}}{u_{\delta,n}}\varphi dx=\int_{\R^{N}} g_{\e_{n}}(x, |u_{n}|^{2})\frac{|u_{n}|^{2}}{u_{\delta,n}}\varphi dx.
\end{align}
Now, using $\Re(z)\leq |z|$ for all $z\in \C$ and  $|e^{\imath t}|=1$ for all $t\in \R$, we get
\begin{align}\label{alves1}
&\Re\left[(u_{n}(x)-u_{n}(y)e^{\imath A_{\e_{n}}(\frac{x+y}{2})\cdot (x-y)}) \left(\frac{\overline{u_{n}(x)}}{u_{\delta,n}(x)}\varphi(x)-\frac{\overline{u_{n}(y)}}{u_{\delta,n}(y)}\varphi(y)e^{-\imath A_{\e_{n}}(\frac{x+y}{2})\cdot (x-y)}  \right)\right] \nonumber\\
&=\Re\left[\frac{|u_{n}(x)|^{2}}{u_{\delta,n}(x)}\varphi(x)+\frac{|u_{n}(y)|^{2}}{u_{\delta,n}(y)}\varphi(y)-\frac{u_{n}(x)\overline{u_{n}(y)}}{u_{\delta,n}(y)}\varphi(y)e^{-\imath A_{\e_{n}}(\frac{x+y}{2})\cdot (x-y)} -\frac{u_{n}(y)\overline{u_{n}(x)}}{u_{\delta,n}(x)}\varphi(x)e^{\imath A_{\e_{n}}(\frac{x+y}{2})\cdot (x-y)}\right] \nonumber \\
&\geq \left[\frac{|u_{n}(x)|^{2}}{u_{\delta,n}(x)}\varphi(x)+\frac{|u_{n}(y)|^{2}}{u_{\delta,n}(y)}\varphi(y)-|u_{n}(x)|\frac{|u_{n}(y)|}{u_{\delta,n}(y)}\varphi(y)-|u_{n}(y)|\frac{|u_{n}(x)|}{u_{\delta,n}(x)}\varphi(x) \right].
\end{align}
Let us note that
\begin{align}\label{alves2}
&\frac{|u_{n}(x)|^{2}}{u_{\delta,n}(x)}\varphi(x)+\frac{|u_{n}(y)|^{2}}{u_{\delta,n}(y)}\varphi(y)-|u_{n}(x)|\frac{|u_{n}(y)|}{u_{\delta,n}(y)}\varphi(y)-|u_{n}(y)|\frac{|u_{n}(x)|}{u_{\delta,n}(x)}\varphi(x) \nonumber\\
&=  \frac{|u_{n}(x)|}{u_{\delta,n}(x)}(|u_{n}(x)|-|u_{n}(y)|)\varphi(x)-\frac{|u_{n}(y)|}{u_{\delta,n}(y)}(|u_{n}(x)|-|u_{n}(y)|)\varphi(y) \nonumber\\
&=\left[\frac{|u_{n}(x)|}{u_{\delta,n}(x)}(|u_{n}(x)|-|u_{n}(y)|)\varphi(x)-\frac{|u_{n}(x)|}{u_{\delta,n}(x)}(|u_{n}(x)|-|u_{n}(y)|)\varphi(y)\right] \nonumber\\
&\quad+\left(\frac{|u_{n}(x)|}{u_{\delta,n}(x)}-\frac{|u_{n}(y)|}{u_{\delta,n}(y)} \right) (|u_{n}(x)|-|u_{n}(y)|)\varphi(y) \nonumber\\
&=\frac{|u_{n}(x)|}{u_{\delta,n}(x)}(|u_{n}(x)|-|u_{n}(y)|)(\varphi(x)-\varphi(y)) +\left(\frac{|u_{n}(x)|}{u_{\delta,n}(x)}-\frac{|u_{n}(y)|}{u_{\delta,n}(y)} \right) (|u_{n}(x)|-|u_{n}(y)|)\varphi(y) \nonumber\\
&\geq \frac{|u_{n}(x)|}{u_{\delta,n}(x)}(|u_{n}(x)|-|u_{n}(y)|)(\varphi(x)-\varphi(y)), 
\end{align}
where in the last inequality we used the fact that
$$
\left(\frac{|u_{n}(x)|}{u_{\delta,n}(x)}-\frac{|u_{n}(y)|}{u_{\delta,n}(y)} \right) (|u_{n}(x)|-|u_{n}(y)|)\varphi(y)\geq 0
$$
because
$$
h(t)=\frac{t}{\sqrt{t^{2}+\delta^{2}}} \mbox{ is increasing for } t\geq 0 \quad \mbox{ and } \quad \varphi\geq 0 \mbox{ in }\R^{N}.
$$
Since
$$
\frac{|\frac{|u_{n}(x)|}{u_{\delta,n}(x)}(|u_{n}(x)|-|u_{n}(y)|)(\varphi(x)-\varphi(y))|}{|x-y|^{N+2s}}\leq \frac{||u_{n}(x)|-|u_{n}(y)||}{|x-y|^{\frac{N+2s}{2}}} \frac{|\varphi(x)-\varphi(y)|}{|x-y|^{\frac{N+2s}{2}}}\in L^{1}(\R^{2N}),
$$
and $\frac{|u_{n}(x)|}{u_{\delta,n}(x)}\rightarrow 1$ a.e. in $\R^{N}$ as $\delta\rightarrow 0$,
we can use \eqref{alves1}, \eqref{alves2} and the Dominated Convergence Theorem to deduce that
\begin{align}\label{Kato2}
&\limsup_{\delta\rightarrow 0} \Re\left[\iint_{\R^{2N}} \frac{(u_{n}(x)-u_{n}(y)e^{\imath A_{\e_{n}}(\frac{x+y}{2})\cdot (x-y)})}{|x-y|^{N+2s}} \left(\frac{\overline{u_{n}(x)}}{u_{\delta,n}(x)}\varphi(x)-\frac{\overline{u_{n}(y)}}{u_{\delta,n}(y)}\varphi(y)e^{-\imath A_{\e_{n}}(\frac{x+y}{2})\cdot (x-y)}  \right) dx dy\right] \nonumber\\
&\geq \limsup_{\delta\rightarrow 0} \iint_{\R^{2N}} \frac{|u_{n}(x)|}{u_{\delta,n}(x)}(|u_{n}(x)|-|u_{n}(y)|)(\varphi(x)-\varphi(y)) \frac{dx dy}{|x-y|^{N+2s}} \nonumber\\
&=\iint_{\R^{2N}} \frac{(|u_{n}(x)|-|u_{n}(y)|)(\varphi(x)-\varphi(y))}{|x-y|^{N+2s}} dx dy.
\end{align}
On the other hand, from the Dominated Convergence Theorem again (we recall that $\frac{|u_{n}|^{2}}{u_{\delta, n}}\leq |u_{n}|$ and $\varphi\in C^{\infty}_{c}(\R^{N}, \R)$) we can see that
\begin{equation}\label{Kato3}
\lim_{\delta\rightarrow 0} \int_{\R^{N}} V_{\e_{n}}(x)\frac{|u_{n}|^{2}}{u_{\delta,n}}\varphi dx=\int_{\R^{N}} V_{\e_{n}}(x)|u_{n}|\varphi dx
\end{equation}
and
\begin{equation}\label{Kato4}
\lim_{\delta\rightarrow 0}  \int_{\R^{N}} g_{\e_{n}}(x, |u_{n}|^{2})\frac{|u_{n}|^{2}}{u_{\delta,n}}\varphi dx=\int_{\R^{N}} g_{\e_{n}}(x, |u_{n}|^{2}) |u_{n}|\varphi dx.
\end{equation}
Putting together \eqref{Kato1}, \eqref{Kato2}, \eqref{Kato3} and \eqref{Kato4} we can deduce that
\begin{align*}
\iint_{\R^{2N}} \frac{(|u_{n}(x)|-|u_{n}(y)|)(\varphi(x)-\varphi(y))}{|x-y|^{N+2s}} dx dy+\int_{\R^{N}} V_{\e_{n}}(x)|u_{n}|\varphi dx\leq 
\int_{\R^{N}} g_{\e_{n}}(x, |u_{n}|^{2}) |u_{n}|\varphi dx
\end{align*}
for any $\varphi\in C^{\infty}_{c}(\R^{N}, \R)$ such that $\varphi\geq 0$, that is $|u_{n}|$ is a weak subsolution to \eqref{Kato0}.
Then, using $(V_{1})$, it is clear that $v_{n}=|u_{n}|(\cdot+\tilde{y}_{n})$ solves 
\begin{equation}\label{Pkat}
(-\Delta)^{s} v_{n} + V_{0}v_{n}\leq g(\e_{n} x+\e_{n}\tilde{y}_{n}, v_{n}^{2})v_{n} \mbox{ in } \R^{N}. 
\end{equation}
Let us denote by $z_{n}\in H^{s}(\R^{N}, \R)$ the unique solution to
\begin{equation}\label{US}
(-\Delta)^{s} z_{n} + V_{0}z_{n}=g_{n} \mbox{ in } \R^{N},
\end{equation}
where
$$
g_{n}:=g(\e_{n} x+\e_{n}\tilde{y}_{n}, v_{n}^{2})v_{n}\in L^{r}(\R^{N}, \R) \quad \forall r\in [2, \infty].
$$
Since \eqref{UBu} yields $\|v_{n}\|_{L^{\infty}(\R^{N})}\leq C$ for all $n\in \mathbb{N}$, by interpolation we know that $v_{n}\rightarrow v$ strongly converges  in $L^{r}(\R^{N}, \R)$ for all $r\in (2, \infty)$, for some $v\in L^{r}(\R^{N}, \R)$, and by the growth assumptions on $g$, we can see that also $g_{n}\rightarrow  f(v^{2})v$ in $L^{r}(\R^{N}, \R)$ and $\|g_{n}\|_{L^{\infty}(\R^{N})}\leq C$ for all $n\in \mathbb{N}$.
From \cite{FQT}, we deduce that $z_{n}=\mathcal{K}*g_{n}$, where $\mathcal{K}$ is the Bessel kernel, and arguing as in \cite{AM}, we deduce that $|z_{n}(x)|\rightarrow 0$ as $|x|\rightarrow \infty$ uniformly with respect to $n\in \mathbb{N}$.
Since $v_{n}$ satisfies \eqref{Pkat} and $z_{n}$ solves \eqref{US}, by comparison it is easy to see that $0\leq v_{n}\leq z_{n}$ a.e. in $\R^{N}$ and for all $n\in \mathbb{N}$. In particular, we can infer that $v_{n}(x)\rightarrow 0$ as $|x|\rightarrow \infty$ uniformly with respect to $n\in \mathbb{N}$.
\end{proof}

\begin{remark}
We recall that in \cite{HIL} the authors proved a Kato's inequality for the fractional magnetic operator $((-\imath\nabla-A(x))^{2}+m^{2})^{\frac{\alpha}{2}}$ with $\alpha\in (0, 1]$ and $m> 0$, or $\alpha=1$ and $m=0$, borrowing some arguments used in \cite{Kato}. As observed in \cite{DS}, when $\alpha= 1$ and $m=0$, this operator coincides with $(-\Delta)^{1/2}_{A}$. 
However, we suspect that a Kato's type-inequality of the form 
$$
(-\Delta)^{s}|u|\leq \Re(sign(u) (-\Delta)^{s}_{A}u)
$$
holds for any $u\in H^{s}_{A}(\R^{N}, \C)$, with $s\in (0,1)$.
Indeed, when $u\in C^{\infty}_{c}(\R^{N}, \C)\setminus\{0\}$ we have  the following pointwise Kato's inequality
\begin{align*}
(-\Delta)^{s}|u|(x)&=\int_{\R^{N}}\frac{|u(x)|-|u(y)|}{|x-y|^{N+2s}} dy=\int_{\R^{N}}\frac{\frac{|u(x)|^{2}}{|u(x)|}-|u(y)|\frac{|\bar{u}(x)|}{|u(x)|}}{|x-y|^{N+2s}} dy \\
&=\int_{\R^{N}}\frac{\frac{|u(x)|^{2}}{|u(x)|}-\frac{|\bar{u}(x) u(y)e^{\imath A(\frac{x+y}{2})\cdot (x-y)}|}{|u(x)|}}{|x-y|^{N+2s}} dy\\
&\leq \Re\left(\int_{\R^{N}} \frac{\frac{|u(x)|^{2}}{|u(x)|}-\frac{\bar{u}(x) u(y)e^{\imath A(\frac{x+y}{2})\cdot (x-y)}}{|u(x)|}}{|x-y|^{N+2s}} dy\right)\\
&=\Re\left(\frac{\bar{u}(x)}{|u(x)|}\int_{\R^{N}}  \left[\frac{u(x)-u(y)e^{\imath A(\frac{x+y}{2})\cdot (x-y)}}{|x-y|^{N+2s}}\right] dy\right) \\
&=\Re(sign(u) (-\Delta)^{s}_{A}u)(x).
\end{align*}
Moreover, for any $u\in H^{s}_{A}(\R^{N}, \C)$ such that $c_{1}\leq |u(x)|\leq c_{2}$ a.e. $x\in \R^{N}$, for some $c_{1}, c_{2}>0$, we can follow the arguments in the proof of the above lemma (it is enough to replace $u_{\delta}$ by $|u|$ and use the fact that $|u|$ is bounded from below and above) to see that
\begin{align}\label{Katofinale}
&\Re\left[\iint_{\R^{2N}} \frac{(u(x)-u(y)e^{\imath A(\frac{x+y}{2})\cdot (x-y)})}{|x-y|^{N+2s}} \left(\frac{\overline{u(x)}}{|u(x)|}\varphi(x)-\frac{\overline{u(y)}}{|u(y)|}\varphi(y)e^{-\imath A(\frac{x+y}{2})\cdot (x-y)}  \right) dx dy\right] \nonumber\\
&\geq \iint_{\R^{2N}} \frac{(|u(x)|-|u(y)|)(\varphi(x)-\varphi(y))}{|x-y|^{N+2s}} dx dy
\end{align}
for any $\varphi\in C^{\infty}_{c}(\R^{N}, \R)$ such that $\varphi\geq 0$. Unfortunately, if $|u|$ does not satisfy the above bounds, we can not use $\frac{u}{|u|}\varphi$ as test function to prove \eqref{Katofinale}. This motives the use of  $u_{\delta}=\sqrt{|u|^{2}+\delta^{2}}$.
 
\end{remark}

\noindent
We end this section giving the proof of Theorem \ref{thm1}.
\begin{proof}
For any $\e_{n}\rightarrow 0$, let $u_{n}\in H^{s}_{\e_{n}}$ be such that $J_{\e_{n}}(u_{n})=c_{\e_{n}}$ and $J'_{\e_{n}}(u_{n})=0$.
Using Lemma \ref{prop3.3}, there exists $(\tilde{y}_{n})\subset \R^{N}$ such that $\e_{n}\tilde{y}_{n}\rightarrow y_{0}$ for some $y_{0} \in \Lambda$ such that $V(y_{0})=V_{0}$. 
Then we can find $r>0$ such that, for some subsequence still denoted by itself, we obtain $B_{r}(\e_{n}\tilde{y}_{n})\subset \Lambda$ for any $n\in \mathbb{N}$.
Therefore, $B_{\frac{r}{\e_{n}}}(\tilde{y}_{n})\subset \Lambda_{\e_{n}}$ for any $n\in \mathbb{N}$. Consequently, 
$$
\R^{N}\setminus \Lambda_{\e_{n}}\subset \R^{N} \setminus B_{\frac{r}{\e_{n}}}(\tilde{y}_{n}) \mbox{ for any } n\in \mathbb{N}.
$$ 
By Lemma \ref{moser}, we can find $R>0$ such that 
$$
v_{n}(x)<a \mbox{ for } |x|\geq R, n\in \mathbb{N},
$$ 
where $v_{n}(x)=|u_{n}|(x+ \tilde{y}_{n})$. 
Hence, $|u_{n}(x)|<a$ for any $x\in \R^{N}\setminus B_{R}(\tilde{y}_{n})$ and $n\in \mathbb{N}$. Then there exists $\nu \in \mathbb{N}$ such that for any $n\geq \nu$ and $r/\e_{n}>R$ it holds 
$$
\R^{N}\setminus \Lambda_{\e_{n}}\subset \R^{N} \setminus B_{\frac{r}{\e_{n}}}(\tilde{y}_{n})\subset \R^{N}\setminus B_{R}(\tilde{y}_{n}),
$$ 
which gives $|u_{n}(x)|<a$ for any $x\in \R^{N}\setminus \Lambda_{\e_{n}}$ and $n\geq \nu$.  \\
This means that there exists $\e_{0}>0$ such that, for all $\e\in (0, \e_{0})$, problem \eqref{R} admits a nontrivial solution $u_{\e}$. 
Taking $\hat{u}_{\e}(x)=u_{\e}(x/\e)$, we can infer that $\hat{u}_{\e}$ is a solution to (\ref{P}). 
Finally, we study the behavior of the maximum points of  $|u_{n}|$. In view of $(g_1)$, there exists $\gamma\in (0,a)$ such that
\begin{align}\label{4.18HZ}
g_{\e}(x, t^{2})t^{2}\leq \frac{V_{0}}{2}t^{2}, \mbox{ for all } x\in \R^{N}, |t|\leq \gamma.
\end{align}
Arguing as above, we can take $R>0$ such that
\begin{align}\label{4.19HZ}
\||u_{n}|\|_{L^{\infty}(\R^{N}\setminus B_{R}(\tilde{y}_{n}))}<\gamma.
\end{align}
Up to a subsequence, we may also assume that
\begin{align}\label{4.20HZ}
\||u_{n}|\|_{L^{\infty}(B_{R}(\tilde{y}_{n}))}\geq \gamma.
\end{align}
Indeed, if \eqref{4.20HZ} is not true, we get $\||u_{n}|\|_{L^{\infty}(\R^{N})}< \gamma$, and from $J_{\e_{n}}'(u_{n})=0$, \eqref{4.18HZ} and Lemma \ref{DI} it follows that 
$$
[|u_{n}|]^{2}+\int_{\R^{N}}V_{0}|u_{n}|^{2}dx\leq \|u_{n}\|^{2}_{\e_{n}}=\int_{\R^{N}} g_{\e_{n}}(x, |u_{n}|^{2})|u_{n}|^{2}\,dx\leq \frac{V_{0}}{2}\int_{\R^{N}}|u_{n}|^{2}\, dx
$$
which gives $\||u_{n}|\|_{0}=0$, that is a contradiction. Hence \eqref{4.20HZ} holds true.\\
Taking into account \eqref{4.19HZ} and \eqref{4.20HZ}, we can infer that the maximum point $p_{n}$ of $|u_{n}|$ belong to $B_{R}(\tilde{y}_{n})$, that is $p_{n}=\tilde{y}_{n}+q_{n}$ for some $q_{n}\in B_{R}$. 
Observing that $\hat{u}_{n}(x)=u_{n}(x/\e_{n})$ is the solution to \eqref{P}, we can see that the maximum point $\eta_{n}$ of $|\hat{u}_{n}|$ is of the form $\eta_{n}=\e_{n}\tilde{y}_{n}+\e_{n}q_{n}$. Since $q_{n}\in B_{R}$, $\e_{n}\tilde{y}_{n}\rightarrow y_{0}$ and $V(y_{0})=V_{0}$, from the continuity of $V$ we can conclude that
$$
\lim_{n\rightarrow \infty} V(\eta_{n})=V_{0}.
$$
Next we give a decay estimate for $|\hat{u}_{n}|$.
Firstly, we recall that in virtue of Lemma $4.3$ in \cite{FQT} there exists a function $w$ such that 
\begin{align}\label{HZ1}
0<w(x)\leq \frac{C}{1+|x|^{N+2s}}
\end{align}
and
\begin{align}\label{HZ2}
(-\Delta)^{s} w+\frac{V_{0}}{2}w\geq 0 \mbox{ in } \R^{N}\setminus B_{R_{1}}, 
\end{align}
for some suitable $R_{1}>0$. Using Lemma \ref{moser}, we know that $v_{n}(x)\rightarrow 0$ as $|x|\rightarrow \infty$ uniformly in $n\in \mathbb{N}$, so there exists $R_{2}>0$ such that
\begin{equation}\label{hzero}
h_{n}:=g(\e_{n}x+\e_{n}\tilde{y}_{n}, v_{n}^{2})v_{n}\leq \frac{V_{0}}{2}v_{n}  \mbox{ in } \R^{N}\setminus B_{R_{2}}.
\end{equation}
Let us denote by $w_{n}$ the unique solution to 
$$
(-\Delta)^{s}w_{n}+V_{0}w_{n}=h_{n} \mbox{ in } \R^{N}.
$$
Then $w_{n}(x)\rightarrow 0$ as $|x|\rightarrow \infty$ uniformly in $n\in \mathbb{N}$, and by comparison $0\leq v_{n}\leq w_{n}$ in $\R^{N}$. Moreover, in light of \eqref{hzero}, it holds
\begin{align*}
(-\Delta)^{s}w_{n}+\frac{V_{0}}{2}w_{n}=h_{n}-\frac{V_{0}}{2}w_{n}\leq 0 \mbox{ in } \R^{N}\setminus B_{R_{2}}.
\end{align*}
Choose $R_{3}=\max\{R_{1}, R_{2}\}$ and we set 
\begin{align}\label{HZ4}
c=\inf_{B_{R_{3}}} w>0 \quad  \mbox{ and } \quad \tilde{w}_{n}=(b+1)w-c w_{n},
\end{align}
where $b=\sup_{n\in \mathbb{N}} \|w_{n}\|_{L^{\infty}(\R^{N})}<\infty$. 
Our goal is to show that 
\begin{equation}\label{HZ5}
\tilde{w}_{n}\geq 0 \mbox{ in } \R^{N}.
\end{equation}
Firstly, we observe that
\begin{align}
&\lim_{|x|\rightarrow \infty} \sup_{n\in \mathbb{N}}\tilde{w}_{n}(x)=0,  \label{HZ0N} \\
&\tilde{w}_{n}\geq bc+w-bc>0 \mbox{ in } B_{R_{3}} \label{HZ0},\\
&(-\Delta)^{s} \tilde{w}_{n}+\frac{V_{0}}{2}\tilde{w}_{n}\geq 0 \mbox{ in } \R^{N}\setminus B_{R_{3}} \label{HZ00}.
\end{align}
Now, we argue by contradiction and assume that there exists a sequence $(\bar{x}_{j, n})\subset \R^{N}$ such that 
\begin{align}\label{HZ6}
\inf_{x\in \R^{N}} \tilde{w}_{n}(x)=\lim_{j\rightarrow \infty} \tilde{w}_{n}(\bar{x}_{j, n})<0. 
\end{align}
Thanks to (\ref{HZ0N}), we can deduce that $(\bar{x}_{j, n})$ is bounded and then, up to subsequence, we may suppose that there exists $\bar{x}_{n}\in \R^{N}$ such that $\bar{x}_{j, n}\rightarrow \bar{x}_{n}$ as $j\rightarrow \infty$. 
Thus, (\ref{HZ6}) becomes
\begin{align}\label{HZ7}
\inf_{x\in \R^{N}} \tilde{w}_{n}(x)= \tilde{w}_{n}(\bar{x}_{n})<0.
\end{align}
Using the minimality of $\bar{x}_{n}$ and the representation formula for the fractional Laplacian \cite{DPV}, we can see that 
\begin{align}\label{HZ8}
(-\Delta)^{s}\tilde{w}_{n}(\bar{x}_{n})=\frac{C(N, s)}{2} \int_{\R^{N}} \frac{2\tilde{w}_{n}(\bar{x}_{n})-\tilde{w}_{n}(\bar{x}_{n}+\xi)-\tilde{w}_{n}(\bar{x}_{n}-\xi)}{|\xi|^{N+2s}} d\xi\leq 0.
\end{align}
Taking into account (\ref{HZ0}) and (\ref{HZ6}), we obtain that $\bar{x}_{n}\in \R^{N}\setminus B_{R_{3}}$.
This together with (\ref{HZ7}) and (\ref{HZ8}) implies that 
$$
(-\Delta)^{s} \tilde{w}_{n}(\bar{x}_{n})+\frac{V_{0}}{2}\tilde{w}_{n}(\bar{x}_{n})<0,
$$
which contradicts (\ref{HZ00}). Therefore (\ref{HZ5}) is established. \\
From (\ref{HZ1}), (\ref{HZ5}) and $v_{n}\leq w_{n}$ we get
\begin{align*}
0\leq v_{n}(x)\leq w_{n}(x)\leq \frac{(b+1)}{c}w(x)\leq \frac{\tilde{C}}{1+|x|^{N+2s}} \mbox{ for all } n\in \mathbb{N}, x\in \R^{N},
\end{align*}
for some constant $\tilde{C}>0$. Recalling the definition of $v_{n}$, we can infer that  
\begin{align*}
|\hat{u}_{n}|(x)&=|u_{n}|\left(\frac{x}{\e_{n}}\right)=v_{n}\left(\frac{x}{\e_{n}}-\tilde{y}_{n}\right) \\
&\leq \frac{\tilde{C}}{1+|\frac{x}{\e_{n}}-\tilde{y}_{n}|^{N+2s}} \\
&=\frac{\tilde{C} \e_{n}^{N+2s}}{\e_{n}^{N+2s}+|x- \e_{n} \tilde{y}_{n}|^{N+2s}} \\
&\leq \frac{\tilde{C} \e_{n}^{N+2s}}{\e_{n}^{N+2s}+|x-\eta_{n}|^{N+2s}} \quad \quad \forall x\in \R^{N}.
\end{align*}
\end{proof}

\section{critical magnetic problem}
\noindent
This section is devoted to the study of the existence and concentration of solutions to (\ref{Pcritico}).
Using the change of variable $u(x)\mapsto u(\e x)$ we can consider the following fractional critical problem
\begin{equation}\label{Rcritico}
(-\Delta)^{s}_{A_{\e}} u + V_{\e}(x)u =  f(|u|^{2})u+|u|^{\2-2}u \mbox{ in } \R^{N}.
\end{equation}
Fix $k>1$ and $a>0$ such that $f(a)+a^{\frac{\2-2}{2}}=\frac{V_{0}}{k}$, and we introduce the functions
$$
\tilde{f}(t):=
\begin{cases}
f(t)+(t^{+})^{\frac{\2-2}{2}}& \text{ if $t \leq a$} \\
\frac{V_{0}}{k}    & \text{ if $t >a$}.
\end{cases}
$$ 
and
$$
g(x, t)=\chi_{\Lambda}(x)(f(t)+(t^{+})^{\frac{\2-2}{2}})+(1-\chi_{\Lambda}(x))\tilde{f}(t).
$$
Let us note that from assumptions $(h_1)$-$(h_4)$, $g$ satisfies the following properties:
\begin{compactenum}[($k_1$)]
\item $\displaystyle{\lim_{t\rightarrow 0} g(x, t)=0}$ uniformly in $x\in \R^{N}$;
\item $g(x, t)\leq f(t)+t^{\frac{\2-2}{2}}$ for all $x\in \R^{N}$ and $t>0$;
\item $(i)$ $0< \frac{\theta}{2} G(x, t)\leq g(x, t)t$ for any $x\in \Lambda$ and $t>0$, \\
$(ii)$ $0\leq  G(x, t)\leq g(x, t)t\leq \frac{V(x)}{k}t$ for any $x\in \R^{N}\setminus \Lambda$ and $t>0$;
\item for any $x\in \Lambda$, the function $t\mapsto g(x,t)$ is increasing for $t>0$, and for any $x\in \R^{N}\setminus \Lambda$ the function $t\mapsto g(x,t)$ is increasing for $t\in (0, a)$.
\end{compactenum}
Thus, we consider the following auxiliary problem 
\begin{equation}\label{Pecritico}
(-\Delta)^{s}_{A_{\e}} u + V_{\e}(x)u =  g_{\e}(x, |u|^{2})u \mbox{ in } \R^{N}, 
\end{equation}
and we look for critical points of the following functional
$$
J_{\e}(u)=\frac{1}{2}\|u\|^{2}_{\e}-\frac{1}{2}\int_{\R^{N}} G_{\e}(x, |u|^{2})\, dx.
$$
Let us consider the autonomous problem associated with \eqref{Rcritico}, that is
\begin{equation}\label{APecritico}
(-\Delta)^{s} u + V_{0}u =  f(u^{2})u+|u|^{\2-2}u \mbox{ in } \R^{N}, 
\end{equation}
and we denote by $I_{0}: H^{s}(\R^{N}, \R)\rightarrow \R$ the corresponding functional
$$
I_{0}(u)=\frac{1}{2}\|u\|^{2}_{0}-\frac{1}{2}\int_{\R^{N}} F(u^{2})\, dx-\frac{1}{\2}\int_{\R^{N}} |u|^{\2}\, dx.
$$
Since many calculations are adaptations to those presented in the two early sections, we will emphasize only the differences between the subcritical and critical case. \\
Let us begin proving that $J_{\e}$ possesses a mountain pass geometry.
\begin{lem}\label{MPGcritico}
\begin{compactenum}[$(i)$]
\item $J_{\e}(0)=0$;
\item there exist $\alpha, \rho>0$ such that $J_{\e}(u)\geq \alpha$ for any $u\in \h$ such that $\|u\|_{\e}=\rho$;
\item there exists $e\in \h$ with $\|e\|_{\e}>\rho$ such that $J_{\e}(e)<0$.
\end{compactenum}
\end{lem}
\begin{proof}
First of all, by $(k_1)$-$(k_2)$ and Theorem \ref{Sembedding}, for any $\delta>0$ there exists $C_{\delta}>0$ such that
$$
J_{\e}(u)\geq \frac{1}{2}\|u\|^{2}_{\e}-\delta \|u\|^{2}_{\e}-C_{\delta} \|u\|^{\2}_{\e}
$$
that is $(i)$ holds.
Secondly, using $(k_3)$, for any $u\in \h\setminus\{0\}$ with $\supp(u)\subset \Lambda_{\e}$ and $t>0$ we have 
\begin{align*}
J_{\e}(tu)&\leq \frac{t^{2}}{2} \|u\|^{2}_{\e}-\frac{1}{2}\int_{\Lambda_{\e}} G_{\e}(x, t^{2}|u|^{2})\, dx \\
&\leq \frac{t^{2}}{2} \|u\|^{2}_{\e}-Ct^{\theta} \int_{\Lambda_{\e}} |u|^{\theta}\, dx + C\rightarrow -\infty \quad \mbox{ as } t\rightarrow \infty. 
\end{align*}
\end{proof}

\noindent
Arguing as in Lemma $4.3$ in \cite{A3} and Proposition $3.2.1$ in \cite{DPMV}, we have the following variant of the Concentration-Compactness Lemma (see also \cite{PP}). Firstly, we recall the definition of tight sequence.
\begin{defn}
We say that a sequence $(u_{n})$ is tight in $\mathcal{D}^{s,2}(\R^{N}, \R)$ if for every $\delta>0$ there exists $R>0$ such that $\int_{\R^{N}\setminus B_{R}} |(-\Delta)^{\frac{s}{2}}u_{n}|^{2} dx\leq \delta$ for any $n\in \mathbb{N}$.
\end{defn}
\begin{lem}\label{CCL}
Let $(u_{n})$ be a bounded tight sequence in $\mathcal{D}^{s,2}(\R^{N}, \R)$ such that $u_{n}\rightharpoonup u$ in $\mathcal{D}^{s,2}(\R^{N}, \R)$. 
Let us assume that
\begin{align}\begin{split}\label{46FS}
&|(-\Delta)^{\frac{s}{2}}u_{n}|^{2}\rightharpoonup \mu,  \\
&|u_{n}|^{\2}\rightharpoonup \nu, 
\end{split}\end{align}
in the sense of measure, where $\mu$ and $\nu$ are two bounded non-negative measures on $\R^{N}$. Then, there exist an  at most a countable set $I$, a family of distinct points $(x_{i})_{i\in I}\subset \R^{N}$ and $(\mu_{i})_{i\in I}, (\nu_{i})_{i\in I}\subset (0, \infty)$ such that
\begin{align}
&\nu=|u|^{\2}+\sum_{i\in I} \nu_{i} \delta_{x_{i}} \label{47FS}\\
&\mu\geq |(-\Delta)^{\frac{s}{2}}u|^{2}+\sum_{i\in I} \mu_{i} \delta_{x_{i}} \label{48FS}.
\end{align}
Moreover, the following relation holds true
\begin{align}\label{50FS}
\mu_{i}\geq S_{*} \nu_{i}^{\frac{2}{2^{*}_{s}}} \quad \forall i\in I.
\end{align}
\end{lem}

\noindent
Now we prove the following compactness result.
\begin{lem}\label{PSccritico}
Let $c\in \R$ be such that $c<\frac{s}{N}S_{*}^{\frac{N}{2s}}$. Then $J_{\e}$ satisfies the Palais-Smale condition at the level $c$.
\end{lem}
\begin{proof}
Let $(u_{n})\subset \h$ be a $(PS)_{c}$ sequence. We note that $(u_{n})$ is bounded because using $(k_3)$ we have
\begin{align*}
c+o_{n}(1)\|u_{n}\|_{\e}&\geq J_{\e}(u_{n})-\frac{1}{\theta}\langle J'_{\e}(u_{n}), u_{n}\rangle \\
&\geq \left(\frac{1}{2}-\frac{1}{\theta}\right)\|u_{n}\|^{2}_{\e}+\frac{1}{\theta}\int_{\R^{N}\setminus \Lambda_{\e}} \left[g_{\e}(x, |u_{n}|^{2})|u_{n}|^{2}-\frac{\theta}{2} G_{\e}(x, |u_{n}|^{2}) \right]\, dx\\
&\geq \left(\frac{\theta-2}{2\theta}\right) \left(1-\frac{1}{k}\right)\|u_{n}\|^{2}_{\e},
\end{align*}
and recalling that $k>1$ we get the thesis.
Then we may assume that $u_{n}\rightharpoonup u$ in $\h$. \\
Since $\langle J'_{\e}(u_{n}), u_{n}\rangle=o_{n}(1)$, we can see that
\begin{align}\label{ionewcritico}
\|u_{n}\|^{2}_{\e}=\int_{\R^{N}} g_{\e}(x, |u_{n}|^{2})|u_{n}|^{2}\, dx+o_{n}(1).
\end{align}
On the other hand, standard calculations show that $u$ is a critical point of $J_{\e}$ and it holds
\begin{align}\label{iocritico}
\|u\|^{2}_{\e}=\int_{\R^{N}} g_{\e}(x, |u|^{2})|u|^{2}\, dx.
\end{align}
Now we aim to show that $(u_{n})$ strongly converges to $u$ in $\h$. \\
In order to achieve our purpose, it is enough to show that $\|u_{n}\|_{\e}\rightarrow \|u\|_{\e}$, that in view of \eqref{ionewcritico} and \eqref{iocritico}, it means to prove that
\begin{equation}\label{ciucciocritico}
\int_{\R^{N}} g_{\e}(x, |u_{n}|^{2})|u_{n}|^{2} \, dx\rightarrow \int_{\R^{N}} g_{\e}(x, |u|^{2})|u|^{2}\, dx.
\end{equation}
We begin proving that for each $\delta>0$ there exists $R=R_{\delta}>0$ such that
\begin{equation}\label{3.3-1critico}
\limsup_{n\rightarrow \infty} \int_{\R^{N}\setminus B_{R}} \int_{\R^{N}}  \frac{|u_{n}(x)-u_{n}(y)e^{\imath A_{\e}(\frac{x+y}{2})\cdot (x-y)}|^{2}}{|x-y|^{N+2s}} dx dy+ \int_{\R^{N}\setminus B_{R}}V_{\e}(x)|u_{n}|^{2} \, dx \leq \delta. 
\end{equation} 
Let $\eta_{R}$ be a cut-off function such that $\eta_{R}=0$ on $B_{R}$, $\eta_{R}=1$ on $\R^{N} \setminus B_{2R}$, $0\leq \eta\leq 1$ and $|\nabla \eta_{R}|\leq \frac{c}{R}$. 
Suppose that $R$ is chosen so that $\Lambda_{\e} \subset B_{R}$.
Since $(u_{n})$ is a bounded (PS) sequence, we have
\begin{equation*}
\langle J_{\e}'(u_{n}), \eta_{R} u_{n}  \rangle = o_{n}(1). 
\end{equation*}
Hence, from $(k_3)$-$(ii)$, we get
\begin{align*}
&\iint_{\R^{2N}} \eta_{R}(x) \frac{|u_{n}(x)-u_{n}(y)e^{\imath A_{\e}(\frac{x+y}{2})\cdot (x-y)}|^{2}}{|x-y|^{N+2s}}  dx dy\\
&\quad +\Re\left(\iint_{\R^{2N}} \frac{(\eta_{R}(x)-\eta_{R}(y))(u_{n}(x)-u_{n}(y)e^{\imath A_{\e}(\frac{x+y}{2})\cdot (x-y)})}{|x-y|^{N+2s}} \overline{u_{n}(y)}e^{-\imath A_{\e}(\frac{x+y}{2})\cdot (x-y)} dx dy\right)\\
& \quad + \int_{\R^{N}} V_{\e}(x) |u_{n}|^{2} \eta_{R} \, dx \\
&=\int_{\R^{N}} g_{\e}(x, |u_{n}|^{2})\eta_{R}|u_{n}|^{2} + o_{n}(1)\leq \frac{1}{k} \int_{\R^{N}\setminus \Lambda_{\e}} V_{\e}(x)|u_{n}|^{2} \, dx + o_{n}(1), 
\end{align*}
which implies that
\begin{align}\begin{split}\label{jtcritico}
& \int_{\R^{N}\setminus B_{R}} \int_{\R^{N}}  \frac{|u_{n}(x)-u_{n}(y)e^{\imath A_{\e}(\frac{x+y}{2})\cdot (x-y)}|^{2}}{|x-y|^{N+2s}} dx dy+ \left(1-\frac{1}{k}\right) \int_{\R^{N}\setminus B_{R}} V_{\e}(x) |u_{n}|^{2} \, dx \\
&\leq -\Re\left(\iint_{\R^{2N}} \frac{(\eta_{R}(x)-\eta_{R}(y))(u_{n}(x)-u_{n}(y)e^{\imath A_{\e}(\frac{x+y}{2})\cdot (x-y)})}{|x-y|^{N+2s}} \overline{u_{n}(y)}e^{-\imath A_{\e}(\frac{x+y}{2})\cdot (x-y)} dx dy\right)+ o_{n}(1).
\end{split}\end{align}
Using the H\"older inequality and the boundedness of $(u_{n})$ in $\h$ we have that
\begin{align*}
&\left| \Re\left(\iint_{\R^{2N}} \frac{(\eta_{R}(x)-\eta_{R}(y))(u_{n}(x)-u_{n}(y)e^{\imath A_{\e}(\frac{x+y}{2})\cdot (x-y)})}{|x-y|^{N+2s}} \overline{u_{n}(y)}e^{-\imath A_{\e}(\frac{x+y}{2})\cdot (x-y)} dx dy\right) \right| \\
&\leq C  \left( \iint_{\R^{2N}}\frac{|\eta_{R}(x)-\eta_{R}(y)|^{2}}{|x-y|^{N+2s}} |u_{n}(y)|^{2} dx dy\right)^{\frac{1}{2}}.
\end{align*} 
Arguing as in the proof of Lemma \ref{PSc} we deduce that
\begin{align}\label{JTcritico}
\limsup_{R\rightarrow \infty} \limsup_{n\rightarrow \infty} \iint_{\R^{2N}}\frac{|\eta_{R}(x)-\eta_{R}(y)|^{2}}{|x-y|^{N+2s}} |u_{n}(x)|^{2} dx dy=0.
\end{align}
Putting together \eqref{jtcritico} and \eqref{JTcritico} we can infer that  \eqref{3.3-1critico} holds.\\
Now, using \eqref{3.3-1critico}, $(k_2)$, $(h_1)$, $(h_2)$ and $H^{s}(\R^{N}, \R)\subset L^{r}(\R^{N}, \R)$ for any $r\in [2, \2]$, we obtain that 
\begin{equation}\label{3.6-1critico}
\int_{\R^{N} \setminus B_{R}} g_{\e}(x, |u_{n}|^{2})|u_{n}|^{2}\, dx \leq \frac{\delta}{4}, 
\end{equation}
for any $n$ big enough.
On the other hand, choosing $R$ large enough, we may assume that
\begin{equation}\label{3.6-2critico}
\int_{\R^{N} \setminus B_{R}} g_{\e}(x, |u|^{2})|u|^{2}\, dx \leq \frac{\delta}{4}. 
\end{equation}
From the arbitrariness of $\delta>0$, we can see that \eqref{3.6-1critico} and \eqref{3.6-2critico} yield
\begin{align}\label{ioo} 
\int_{\R^{N} \setminus B_{R}} &g_{\e}(x, |u_{n}|^{2})|u_{n}|^{2}\, dx \rightarrow \int_{\R^{N} \setminus B_{R}} g_{\e} (x, |u|^{2})|u|^{2}\, dx
\end{align}
as $n\rightarrow \infty$.
Using the definition of $g$ it follows that
$$
g_{\e}(x, |u_{n}|^{2})|u_{n}|^{2}\leq f(|u_{n}|^{2})|u_{n}|^{2}+a^{\2}+\frac{V_{0}}{k}|u_{n}|^{2} \mbox{ for any } x\in \R^{N}\setminus \Lambda_{\e}.
$$
Since $B_{R}\cap (\R^{N}\setminus \Lambda_{\e})$ is bounded, we can use the above estimate, $(h_1)$, $(h_2)$, Theorem \ref{Sembedding} and the Dominated Convergence Theorem to infer that, as $n\rightarrow \infty$, 
\begin{align}\label{iooocritico}
\int_{B_{R}\cap (\R^{N}\setminus \Lambda_{\e})} &g_{\e}(x, |u_{n}|^{2})|u_{n}|^{2}\, dx \rightarrow \int_{B_{R}\cap (\R^{N}\setminus \Lambda_{\e})} g_{\e}(x, |u|^{2})|u|^{2}\, dx. 
\end{align}
At this point, we aim to show that
\begin{equation}\label{ioooocritico}
\lim_{n\rightarrow \infty}\int_{\Lambda_{\e}} |u_{n}|^{\2} \,dx=\int_{\Lambda_{\e}} |u|^{\2} \,dx.
\end{equation}
Indeed, if we assume that \eqref{ioooocritico} is true, from $(k_2)$, $(h_1)$, $(h_2)$, Theorem \ref{Sembedding} and the Dominated Convergence Theorem we deduce that
\begin{equation}\label{i5critico}
\int_{B_{R}\cap \Lambda_{\e}} g_{\e}(x, |u_{n}|^{2})|u_{n}|^{2}\, dx\rightarrow \int_{B_{R}\cap \Lambda_{\e}} g_{\e}(x, |u|^{2})|u|^{2}\, dx.
\end{equation}
Putting together \eqref{ioo},\eqref{iooocritico} and \eqref{i5critico} we can conclude that \eqref{ciucciocritico} holds. \\
In what follows we prove that \eqref{ioooocritico} is satisfied. From \eqref{3.3-1critico} and Lemma \ref{DI}, we can see that $(|u_{n}|)$ is tight in $H^{s}(\R^{N}, \R)$ and so, 
by Lemma \ref{CCL}, we can find an at most countable index set $I$, sequences $(x_{i})_{i\in I}\subset \R^{N}$, $(\mu_{i})_{i\in I}, (\nu_{i})_{i\in I}\subset (0, \infty)$ such that 
\begin{align}\label{CML}
&\mu\geq |(-\Delta)^{\frac{s}{2}}|u||^{2}+\sum_{i\in I} \mu_{i} \delta_{x_{i}}, \nonumber \\
&\nu=|u|^{\2}+\sum_{i\in I} \nu_{i} \delta_{x_{i}} \quad \mbox{ and } \quad S_{*} \nu_{i}^{\frac{2}{2^{*}_{s}}}\leq \mu_{i} \quad \forall i\in I,
\end{align}
where $\delta_{x_{i}}$ is the Dirac mass at the point $x_{i}$.
Let us show that $(x_{i})_{i\in I}\cap \Lambda_{\e}=\emptyset$. Assume by contradiction that 
$x_{i}\in \Lambda_{\e}$ for some $i\in I$. For any $\rho>0$, we define $\psi_{\rho}(x)=\psi(\frac{x-x_{i}}{\rho})$ where $\psi\in C^{\infty}_{c}(\R^{N}, [0, 1])$ is such that $\psi=1$ in $B_{1}$, $\psi=0$ in $\R^{N}\setminus B_{2}$ and $\|\nabla \psi\|_{L^{\infty}(\R^{N})}\leq 2$. We suppose that  $\rho>0$ is such that $\supp(\psi_{\rho})\subset \Lambda_{\e}$. Since $(\psi_{\rho} u_{n})$ is bounded in $\h$, we can see that $\langle J'_{\e}(u_{n}),\psi_{\rho} u_{n}\rangle=o_{n}(1)$, and  using the pointwise diamagnetic inequality in Lemma \ref{DI} we get
\begin{align}\label{amicicritico}
\iint_{\R^{2N}} &\psi_{\rho}(y)\frac{||u_{n}(x)|-|u_{n}(y)||^{2}}{|x-y|^{N+2s}} \, dx dy \nonumber\\
&\leq -\Re\left(\iint_{\R^{2N}} \frac{(\psi_{\rho}(x)-\psi_{\rho}(y))(u_{n}(x)-u_{n}(y)e^{\imath A_{\e}(\frac{x+y}{2})\cdot (x-y)})}{|x-y|^{N+2s}} \overline{u_{n}(y)}e^{-\imath A_{\e}(\frac{x+y}{2})\cdot (x-y)} dx dy\right) \nonumber\\
&\quad +\int_{\R^{N}} \psi_{\rho} f(|u_{n}|^{2})|u_{n}|^{2}\, dx+\int_{\R^{N}} \psi_{\rho} |u_{n}|^{\2}\, dx+o_{n}(1).
\end{align}
Due to the fact that $f$ has subcritical growth and $\psi_{\rho}$ has compact support, we have that
\begin{align}\label{Alessia1critico}
\lim_{\rho\rightarrow 0} &\lim_{n\rightarrow \infty} \int_{\R^{N}}  \psi_{\rho} f(|u_{n}|^{2})|u_{n}|^{2}\, dx=\lim_{\rho\rightarrow 0} \int_{\R^{N}}  \psi_{\rho} f(|u|^{2})|u|^{2}\, dx=0.
\end{align}
Now we show that
\begin{equation}\label{niocritico}
\lim_{\rho\rightarrow 0}\lim_{n\rightarrow \infty} \Re\left(\iint_{\R^{2N}} \frac{(\psi_{\rho}(x)-\psi_{\rho}(y))(u_{n}(x)-u_{n}(y)e^{\imath A_{\e}(\frac{x+y}{2})\cdot (x-y)})}{|x-y|^{N+2s}} \overline{u_{n}(y)}e^{-\imath A_{\e}(\frac{x+y}{2})\cdot (x-y)} dx dy\right)=0.
\end{equation}
Using the H\"older inequality and the fact that $(u_{n})$ is bounded in $\h$, we can see that
\begin{align*}
&\left|\Re\left(\iint_{\R^{2N}} \frac{(\psi_{\rho}(x)-\psi_{\rho}(y))(u_{n}(x)-u_{n}(y)e^{\imath A_{\e}(\frac{x+y}{2})\cdot (x-y)})}{|x-y|^{N+2s}} \overline{u_{n}(y)}e^{-\imath A_{\e}(\frac{x+y}{2})\cdot (x-y)} dx dy\right) \right|  \\
&\leq C\left(\iint_{\R^{2N}} |u_{n}(y)|^{2} \frac{|\psi_{\rho}(x)-\psi_{\rho}(y)|^{2}}{|x-y|^{N+2s}} \, dx dy \right)^{\frac{1}{2}}.
\end{align*}
Therefore, it is enough to verify that
\begin{equation}\label{NIOOcritico}
\lim_{\rho\rightarrow 0}\lim_{n\rightarrow \infty}  \iint_{\R^{2N}} |u_{n}(x)|^{2} \frac{|\psi_{\rho}(x)-\psi_{\rho}(y)|^{2}}{|x-y|^{N+2s}} \, dx dy =0
\end{equation}
to deduce that  \eqref{niocritico} holds. \\
Firstly, we write $\R^{2N}$  as 
\begin{align*}
\R^{2N}&=((\R^{N}\setminus B_{2\rho}(x_{i}))\times (\R^{N} \setminus B_{2\rho}(x_{i})))\cup (B_{2\rho}(x_{i})\times \R^{N}) \cup ((\R^{N}\setminus B_{2\rho}(x_{i}))\times B_{2\rho}(x_{i}))\\
&=: X^{1}_{\rho}\cup X^{2}_{\rho} \cup X^{3}_{\rho}.
\end{align*}
Accordingly,
\begin{align}\label{Pa1critico}
&\iint_{\R^{2N}}  |u_{n}(x)|^{2} \frac{(\psi_{\rho}(x)-\psi_{\rho}(y))^{2}}{|x-y|^{N+2s}} \, dx dy \nonumber\\
&=\iint_{X^{1}_{\rho}}  |u_{n}(x)|^{2}\frac{|\psi_{\rho}(x)-\psi_{\rho}(y)|^{2}}{|x-y|^{N+2s}} \, dx dy +\iint_{X^{2}_{\rho}}  |u_{n}(x)|^{2}\frac{|\psi_{\rho}(x)-\psi_{\rho}(y)|^{2}}{|x-y|^{N+2s}} \, dx dy \nonumber\\
&\quad+ \iint_{X^{3}_{\rho}}  |u_{n}(x)|^{2}\frac{|\psi_{\rho}(x)-\psi_{\rho}(y)|^{2}}{|x-y|^{N+2s}} \, dx dy.
\end{align}
In what follows, we estimate each integral in (\ref{Pa1critico}).
Since $\psi=0$ in $\R^{N}\setminus B_{2}$, we have
\begin{align}\label{Pa2critico}
\iint_{X^{1}_{\rho}}  |u_{n}(x)|^{2}\frac{|\psi_{\rho}(x)-\psi_{\rho}(y)|^{2}}{|x-y|^{N+2s}} \, dx dy=0.
\end{align}
Since $0\leq \psi\leq 1$ and $\|\nabla \psi_{\rho}\|_{L^{\infty}(\R^{N})}\leq C/\rho$, we obtain
\begin{align}\label{Pa3critico}
&\iint_{X^{2}_{\rho}}  |u_{n}(x)|^{2}\frac{|\psi_{\rho}(x)-\psi_{\rho}(y)|^{2}}{|x-y|^{N+2s}} \, dx dy \nonumber\\
&=\int_{B_{2\rho}(x_{i})} \,dx \int_{\{y\in \R^{N}: |x-y|\leq \rho\}}  |u_{n}(x)|^{2}\frac{|\psi_{\rho}(x)-\psi_{\rho}(y)|^{2}}{|x-y|^{N+2s}} \, dy \nonumber \\
&\quad +\int_{B_{2\rho}(x_{i})} \, dx \int_{\{y\in \R^{N}: |x-y|> \rho\}} |u_{n}(x)|^{2}\frac{|\psi_{\rho}(x)-\psi_{\rho}(y)|^{2}}{|x-y|^{N+2s}} \, dy  \nonumber\\
&\leq C\rho^{-2} \int_{B_{2\rho}(x_{i})} \, dx \int_{\{y\in \R^{N}: |x-y|\leq \rho\}} \frac{ |u_{n}(x)|^{2}}{|x-y|^{N+2s-2}} \, dy\nonumber \\
&\quad + C \int_{B_{2\rho}(x_{i})} \, dx \int_{\{y\in \R^{N}: |x-y|> \rho\}} \frac{ |u_{n}(x)|^{2}}{|x-y|^{N+2s}} \, dy \nonumber\\
&\leq C \rho^{-2s} \int_{B_{2\rho}(x_{i})}  |u_{n}(x)|^{2} \, dx+C \rho^{-2s} \int_{B_{2\rho}(x_{i})}  |u_{n}(x)|^{2} \, dx =C \rho^{-2s} \int_{B_{2\rho}(x_{i})}  |u_{n}(x)|^{2} \, dx,
\end{align}
for some $C>0$ independent of $n$.
On the other hand, 
\begin{align}\label{Pa4critico}
&\iint_{X^{3}_{\rho}}  |u_{n}(x)|^{2}\frac{|\psi_{\rho}(x)-\psi_{\rho}(y)|^{2}}{|x-y|^{N+2s}} \, dx dy \nonumber\\
&=\int_{\R^{N}\setminus B_{2\rho}(x_{i})} \, dx \int_{\{y\in B_{2\rho}(x_{i}): |x-y|\leq \rho\}}  |u_{n}(x)|^{2}\frac{|\psi_{\rho}(x)-\psi_{\rho}(y)|^{2}}{|x-y|^{N+2s}} \, dy \nonumber\\
&\quad+\int_{\R^{N}\setminus B_{2\rho}(x_{i})} \,dx \int_{\{y\in B_{2\rho}(x_{i}): |x-y|> \rho\}}  |u_{n}(x)|^{2}\frac{|\psi_{\rho}(x)-\psi_{\rho}(y)|^{2}}{|x-y|^{N+2s}} \, dy=: A_{\rho, n}+ B_{\rho, n}. 
\end{align}
Let us note that $|x-y|<\rho$ and $|y-x_{i}|<2\rho$ imply $|x-x_{i}|<3\rho$, and then
\begin{align}\label{Pa5critico}
A_{\rho, n}&\leq C\rho^{-2}  \int_{B_{3\rho}(x_{i})} \, dx \int_{\{y\in B_{2\rho}(x_{i}): |x-y|\leq \rho\}} \frac{ |u_{n}(x)|^{2}}{|x-y|^{N+2s-2}} \, dy \nonumber\\
&\leq C\rho^{-2}  \int_{B_{3\rho}(x_{i})}  |u_{n}(x)|^{2} \, dx \int_{\{z\in \R^{N}: |z|\leq \rho\}} \frac{1}{|z|^{N+2s-2}} \, dz \nonumber\\
&=C \rho^{-2s} \int_{B_{3\rho}(x_{i})}  |u_{n}(x)|^{2} \, dx.
\end{align}
Let us observe that  for all $K>4$ it holds 
$$
(\R^{N}\setminus B_{2\rho}(x_{i}))\times B_{2\rho}(x_{i}) \subset (B_{K\rho}(x_{i})\times B_{2\rho}(x_{i})) \cup ((\R^{N}\setminus B_{K\rho}(x_{i}))\times B_{2\rho}(x_{i})).
$$
Therefore, we can see that
\begin{align}\label{Pa6critico}
&\int_{B_{K\rho}(x_{i})} \, dx \int_{\{y\in B_{2\rho}(x_{i}): |x-y|> \rho\}}  |u_{n}(x)|^{2}\frac{|\psi_{\rho}(x)-\psi_{\rho}(y)|^{2}}{|x-y|^{N+2s}} \, dy \nonumber\\
&\leq C\int_{B_{K\rho}(x_{i})} \, dx \int_{\{y\in B_{2\rho}(x_{i}): |x-y|> \rho\}}  \frac{|u_{n}(x)|^{2}}{|x-y|^{N+2s}} \, dy \nonumber \\
&\leq C  \int_{B_{K\rho}(x_{i})}  |u_{n}(x)|^{2} \, dx \int_{\{z\in \R^{N}: |z|> \rho\}} \frac{1}{|z|^{N+2s}} \, dz \nonumber\\
&= C \rho^{-2s} \int_{B_{K\rho}(x_{i})}  |u_{n}(x)|^{2} \, dx.
\end{align}

On the other hand, if $|x-x_{i}|\geq K\rho$ and $|y-x_{i}|<2\rho$ then 
$$
|x-y|\geq |x-x_{i}|-|y-x_{i}|\geq \frac{|x-x_{i}|}{2}+\frac{K\rho}{2}-2\rho>\frac{|x-x_{i}|}{2}.
$$
Consequently,
\begin{align}\label{Pa7critico}
&\int_{\R^{N}\setminus B_{K\rho}(x_{i})} \, dx \int_{\{y\in B_{2\rho}(x_{i}): |x-y|>\rho\}}  |u_{n}(x)|^{2} \frac{|\psi_{\rho}(x)-\psi_{\rho}(y)|^{2}}{|x-y|^{N+2s}} \, dy \nonumber\\
&\leq  C \int_{\R^{N}\setminus B_{K\rho}(x_{i})} \, dx \int_{\{y\in B_{2\rho}(x_{i}): |x-y|>\rho\}} \frac{ |u_{n}(x)|^{2}}{|x-x_{i}|^{N+2s}} \, dy \nonumber\\
&\leq C \rho^{N} \int_{\R^{N}\setminus B_{K\rho}(x_{i})} \frac{ |u_{n}(x)|^{2}}{|x-x_{i}|^{N+2s}} \, dx \nonumber\\
&\leq C\! \rho^{N}\!\left(\int_{\R^{N}\setminus B_{K\rho}(x_{i})} \!\!\! |u_{n}(x)|^{\2} \, dx\right)^{\frac{2}{2^{*}_{s}}}\!\! \left(\int_{\R^{N}\setminus B_{K\rho}(x_{i})} \!\!\!|x-x_{i}|^{-(N+2s)\frac{2^{*}_{s}}{2^{*}_{s}-2}} \, dx\right)^{\frac{2^{*}_{s}-2}{2^{*}_{s}}} \nonumber\\
&\leq C K^{-N} \left(\int_{\R^{N}\setminus B_{K\rho}(x_{i})}  |u_{n}(x)|^{\2} \, dx\right)^{\frac{2}{2^{*}_{s}}}.
\end{align}
Putting together (\ref{Pa6critico}) and (\ref{Pa7critico}), and using the fact that $(|u_{n}|)$ is bounded in $L^{2^{*}_{s}}(\R^{N}, \R)$, we can find $C>0$  independent of $n$ such that 
\begin{align}\label{Pa8critico}
B_{\rho, n}\leq C \rho^{-2s} \int_{B_{K\rho}(x_{i})}  |u_{n}(x)|^{2} \, dx+C K^{-N}.
\end{align}
Then, (\ref{Pa1critico})-(\ref{Pa5critico}) and (\ref{Pa8critico}) yield
\begin{align}\label{stimacritico}
\iint_{\R^{2N}}  |u_{n}(x)|^{2}\frac{|\psi_{\rho}(x)-\psi_{\rho}(y)|^{2}}{|x-y|^{N+2s}} \, dx dy \leq C \rho^{-2s} \int_{B_{K\rho}(x_{i})}  |u_{n}(x)|^{2} \, dx+C K^{-N}.
\end{align}
Recalling that $|u_{n}|\rightarrow |u|$ strongly in $L^{2}_{loc}(\R^{N}, \R)$ we have
\begin{align*}
\lim_{n\rightarrow \infty}  C \rho^{-2s} \int_{B_{K\rho}(x_{i})}  |u_{n}(x)|^{2} \, dx+C K^{-N} =C \rho^{-2s} \int_{B_{K\rho}(x_{i})} |u(x)|^{2} \, dx+C K^{-N}.
\end{align*}
Using the H\"older inequality we can see that
\begin{align*}
C \rho^{-2s} & \int_{B_{K\rho}(x_{i})}  |u(x)|^{2} \, dx+C K^{-N} \\
&\leq C \rho^{-2s} \left(\int_{B_{K\rho}(x_{i})}  |u(x)|^{2^{*}_{s}} \, dx\right)^{\frac{2}{2^{*}_{s}}} |B_{K\rho}(x_{i})|^{1-\frac{2}{2^{*}_{s}}}+C K^{-N} \\
&\leq C K^{2s}  \left(\int_{B_{K\rho}(x_{i})}  |u(x)|^{2^{*}_{s}} \, dx\right)^{\frac{2}{2^{*}_{s}}}+C K^{-N}\rightarrow C K^{-N} \mbox{ as } \rho\rightarrow 0.
\end{align*}
Hence, 
\begin{align*}
&\lim_{\rho\rightarrow 0} \limsup_{n\rightarrow \infty} \iint_{\R^{2N}}  |u_{n}(x)|^{2}\frac{|\psi_{\rho}(x)-\psi_{\rho}(y)|^{2}}{|x-y|^{N+2s}} \, dx dy
 \nonumber\\
&=\lim_{K\rightarrow \infty}\lim_{\rho\rightarrow 0} \limsup_{n\rightarrow \infty} \iint_{\R^{2N}}  |u_{n}(x)|^{2}\frac{|\psi_{\rho}(x)-\psi_{\rho}(y)|^{2}}{|x-y|^{N+2s}} \, dx dy =0,
\end{align*}
that is \eqref{NIOOcritico} holds.
Therefore, using \eqref{CML} and taking the limit as $n\rightarrow \infty$ and $\rho\rightarrow 0$ in \eqref{amicicritico}, we can deduce that \eqref{Alessia1critico} and \eqref{niocritico} yield $\nu_{i}\geq \mu_{i}$.
From the last statement in \eqref{CML} it follows that $\nu_{i}\geq S_{*}^{\frac{2}{2^{*}_{s}}}$, and using $(h_4)$ and $(k_3)$ we get
\begin{align*}
c&=J_{\e}(u_{n})-\frac{1}{2}\langle J'_{\e}(u_{n}), u_{n}\rangle+o_{n}(1) \\
&=\int_{\R^{N}\setminus \Lambda_{\e}} \frac{1}{2}\left[ g_{\e}(x, |u_{n}|^{2})|u_{n}|^{2}-G_{\e}(x, |u_{n}|^{2})\right] \, dx +\int_{\Lambda_{\e}} \frac{1}{2}\left[f(|u_{n}|^{2})|u_{n}|^{2}-F(|u_{n}|^{2})\right] \, dx \\
&\quad+\frac{s}{N}\int_{\Lambda_{\e}} |u_{n}|^{\2} \, dx+o_{n}(1) \\
&\geq \frac{s}{N}\int_{\Lambda_{\e}} |u_{n}|^{\2} \, dx+o_{n}(1) \\
&\geq \frac{s}{N} \int_{\Lambda_{\e}} \psi_{\rho} |u_{n}|^{\2} \, dx+o_{n}(1).
\end{align*}
Then, using \eqref{CML} and taking the limit as $n\rightarrow \infty$ we find
\begin{align*}
c\geq \frac{s}{N} \sum_{\{i\in I: x_{i}\in \Lambda_{\e}\}} \psi_{\rho}(x_{i}) \nu_{i} =\frac{s}{N} \sum_{\{i\in I: x_{i}\in \Lambda_{\e}\}} \nu_{i} \geq \frac{s}{N} S^{\frac{N}{2s}}_{*},
\end{align*}
which gives a contradiction. This ends the proof of \eqref{ioooocritico}.
\end{proof}

\noindent
Let us define the mountain pass level
$$
c_{\e}=\inf_{\gamma\in \Gamma_{\e}} \max_{t\in [0, 1]} J_{\e}(\gamma(t))
$$
where
$$
\Gamma_{\e}=\{\gamma\in C([0, 1], \h): \gamma(0)=0 \mbox{ and } J_{\e}(\gamma(1))<0\}.
$$
We also denote by $c_{0}$ the mountain pass level associated with $I_{0}$.

Let $w\in H^{s}(\R^{N}, \R)$ be a positive ground state solution for \eqref{APecritico} such that $I'_{0}(w)=0$ and $I_{0}(w)=c_{0}<\frac{s}{N} S_{*}^{\frac{N}{2s}}$ (see \cite{HZ}). Since any solution of \eqref{APecritico} belongs to $C^{0, \alpha}(\R^{N}, \R)\cap L^{2^{*}_{s}}(\R^{N}, \R)$, we know that it goes to zero at infinity. Then we can proceed as in \cite{FQT} to see that the following estimate holds
\begin{equation}\label{remdecay}
0<w(x)\leq \frac{C}{|x|^{N+2s}} \mbox{ for all } |x|>1.
\end{equation}
Arguing as in the proof of Lemma \ref{AMlem1} we can see that the following result holds:
\begin{lem}\label{AMlem1critico}
The numbers $c_{\e}$ and $c_{0}$ satisfy the following relation
$$
\limsup_{\e\rightarrow 0} c_{\e}\leq c_{0}<\frac{s}{N} S_{*}^{\frac{N}{2s}}.
$$
\end{lem}


\noindent
Let us recall the following result for the autonomous problem \eqref{APecritico} (see \cite{HZ}).
\begin{lem}\label{FScritico}
Let $(u_{n})\subset \mathcal{N}_{0}$ be a sequence satisfying $I_{0}(u_{n})\rightarrow c<\frac{s}{N} S_{*}^{\frac{N}{2s}}$. Then, up to subsequences, one of the following alternatives holds:
\begin{compactenum}[(i)]
\item $(u_{n})$ strongly converges in $H^{s}(\R^{N}, \R)$, 
\item there exists a sequence $(\tilde{y}_{n})\subset \R^{N}$ such that,  up to a subsequence, $v_{n}(x)=u_{n}(x+\tilde{y}_{n})$ strongly converges  in $H^{s}(\R^{N}, \R)$.
\end{compactenum}
\end{lem}

\begin{lem}\label{prop3.3critico}
Let $\e_{n}\rightarrow 0$ and $u_{n}\in H^{s}_{\e_{n}}$ be such that $J_{\e_{n}}(u_{n})=c_{\e_{n}}$ and $J'_{\e_{n}}(u_{n})=0$. Then there exists $(\tilde{y}_{n})\subset \R^{N}$ such that $v_{n}(x)=|u_{n}|(x+\tilde{y}_{n})$ has a convergent subsequence in $H^{s}(\R^{N}, \R)$. Moreover, up to a subsequence, $y_{n}=\e_{n} \tilde{y}_{n}\rightarrow y_{0}$ for some $y_{0}\in \Lambda$ such that $V(y_{0})=V_{0}$.
\end{lem}
\begin{proof}
From $\langle J'_{\e_{n}}(u_{n}), u_{n}\rangle=0$ and Lemma \ref{AMlem1critico}, it follows that $(u_{n})$ is bounded in $H^{s}_{\e_{n}}$, so there exists $C>0$ (independent of $n$) such that $\|u_{n}\|_{\e_{n}}\leq C$ for all $n\in \mathbb{N}$.\\
Now we prove that there exist a sequence $(\tilde{y}_{n})\subset \R^{N}$ and constants $R>0$ and $\gamma>0$ such that
\begin{equation}\label{sacchicritico}
\liminf_{n\rightarrow \infty}\int_{B_{R}(\tilde{y}_{n})} |u_{n}|^{2} \, dx\geq \gamma>0.
\end{equation}
Suppose by contradiction that condition \eqref{sacchicritico} does not hold. Then, for all $R>0$ we have
$$
\lim_{n\rightarrow \infty}\sup_{y\in \R^{N}}\int_{B_{R}(y)} |u_{n}|^{2} \, dx=0.
$$
Since we know that $(|u_{n}|)$ is bounded in $H^{s}(\R^{N}, \R)$, we can use Lemma $2.2$ in \cite{FQT} to deduce that $|u_{n}|\rightarrow 0$ in $L^{q}(\R^{N}, \R)$ for any $q\in (2, 2^{*}_{s})$. 
In particular, by $(h_1)$ and $(h_2)$ it follows that
\begin{align*}
\int_{\R^{N}} F(|u_{n}|^{2}) dx=\int_{\R^{N}} f(|u_{n}|^{2})|u_{n}|^{2} dx=o_{n}(1).
\end{align*}
This implies that 
\begin{align}\label{2.11HZ}
\frac{1}{2}\int_{\R^{N}}  G_{\e_{n}}(x, |u_{n}|^{2}) dx\leq \frac{1}{2^{*}_{s}} \int_{\Lambda_{\e_{n}}\cup \{|u_{n}|^{2}\leq a\}} |u_{n}|^{2^{*}_{s}} dx+\frac{V_{0}}{2K} \int_{\Lambda^{c}_{\e_{n}}\cap \{|u_{n}|^{2}> a\}} |u_{n}|^{2} dx+o_{n}(1)
\end{align}
and
\begin{align}\label{2.12HZ}
\int_{\R^{N}}  g_{\e_{n}}(x, |u_{n}|^{2})|u_{\e_{n}}|^{2} dx= \int_{\Lambda_{\e_{n}}\cup \{|u_{n}|^{2}\leq a\}} |u_{n}|^{2^{*}_{s}} dx+\frac{V_{0}}{K} \int_{\Lambda^{c}_{\e_{n}}\cap \{|u_{n}|^{2}> a\}} |u_{n}|^{2} dx+o_{n}(1),
\end{align}
where we used the notation $\Lambda_{\e}^{c}=\R^{N}\setminus \Lambda_{\e}$.\\
Taking into account $\langle J'_{\e_{n}}(u_{n}), u_{n}\rangle=0$ and \eqref{2.12HZ}, we can deduce that 
\begin{align}\label{2.13HZ}
\|u_{n}\|^{2}_{\e_{n}}-\frac{V_{0}}{K} \int_{\Lambda^{c}_{\e_{n}}\cap \{|u_{n}|^{2}> a\}} |u_{n}|^{2} dx=\int_{\Lambda_{\e_{n}}\cup \{|u_{n}|^{2}\leq a\}} |u_{n}|^{2^{*}_{s}} dx+o_{n}(1).
\end{align}
Let $\ell\geq 0$ be such that 
$$
\|u_{n}\|^{2}_{\e_{n}}-\frac{V_{0}}{K} \int_{\Lambda^{c}_{\e_{n}}\cap \{|u_{n}|^{2}> a\}} |u_{n}|^{2} dx\rightarrow \ell.
$$
It is easy to see that $\ell>0$, otherwise $u_{n}\rightarrow 0$ in $H^{s}_{\e_{n}}$ and this is impossible because $\langle J'_{\e_{n}}(u_{n}), u_{n}\rangle=0$, $(k_1)$ and $(k_2)$ imply that there exists $\alpha_{0}>0$ such that $\|u_{n}\|^{2}_{\e_{n}}\geq \alpha_{0}$ for all $n\in \mathbb{N}$.
 From \eqref{2.13HZ} it follows that 
$$
\int_{\Lambda_{\e_{n}}\cup \{|u_{n}|^{2}\leq a\}} |u_{n}|^{2^{*}_{s}} dx\rightarrow \ell.
$$
Using $J_{\e_{n}}(u_{n})-\frac{1}{\2}\langle J'_{\e}(u_{n}), u_{n}\rangle=c_{\e_{n}}$, \eqref{2.11HZ}, \eqref{2.12HZ} and \eqref{2.13HZ} we can see that $\frac{s}{N}\ell\leq c_{\e_{n}}+o_{n}(1)$.
Now, from the definition of $S_{*}$, we obtain that
$$
\|u_{n}\|^{2}_{\e_{n}}-\frac{V_{0}}{K} \int_{\Lambda^{c}_{\e_{n}}\cap \{|u_{n}|^{2}> a\}} |u_{n}|^{2} dx\geq S_{*} \left(\int_{\Lambda_{\e_{n}}\cup \{|u_{n}|^{2}\leq a\}} |u_{n}|^{2^{*}_{s}} dx  \right)^{\frac{2}{2^{*}_{s}}},
$$
and taking the limit as $n\rightarrow \infty$ we can infer that $\ell\geq S^{\frac{N}{2s}}_{*}$. Therefore, we can deduce that 
$\liminf_{n\rightarrow \infty}c_{\e_{n}}\geq \frac{s}{N} S_{*}^{\frac{N}{2s}}$ which contradicts Lemma \ref{AMlem1critico}. 

Now, we set $v_{n}(x)=|u_{n}|(x+\tilde{y}_{n})$. Then, $(v_{n})$ is bounded in $H^{s}(\R^{N}, \R)$, and we may assume that 
$v_{n}\rightharpoonup v\not\equiv 0$ in $H^{s}(\R^{N}, \R)$  as $n\rightarrow \infty$.
Fix $t_{n}>0$ such that $\tilde{v}_{n}=t_{n} v_{n}\in \mathcal{N}_{0}$. Using Lemma \ref{DI}, we can see that 
$$
c_{0}\leq I_{0}(\tilde{v}_{n})\leq \max_{t\geq 0}J_{\e_{n}}(tu_{n})= J_{\e_{n}}(u_{n})
$$
which together with Lemma \ref{AMlem1critico} gives $I_{0}(\tilde{v}_{n})\rightarrow c_{0}$.
In particular, $\tilde{v}_{n}\nrightarrow 0$ in $H^{s}(\R^{N}, \R)$ and  $t_{n}\rightarrow t^{*}$ for some $t^{*}>0$. Then we have $\tilde{v}_{n}\rightharpoonup \tilde{v}=t^{*}v\not\equiv 0$ in $H^{s}(\R^{N}, \R)$, and from Lemma \ref{FScritico} it follows that
\begin{equation}\label{elenacritico}
\tilde{v}_{n}\rightarrow \tilde{v} \mbox{ in } H^{s}(\R^{N}, \R).
\end{equation} 
In particular, $v_{n}\rightarrow v$ in $H^{s}(\R^{N}, \R)$ as $n\rightarrow \infty$.

In order to complete the proof of lemma, we consider $y_{n}=\e_{n}\tilde{y}_{n}$. Our claim is to show that $(y_{n})$ admits a subsequence, still denoted by $y_{n}$, such that $y_{n}\rightarrow y_{0}$, for some $y_{0}\in \Lambda$ such that $V(y_{0})=V_{0}$. Firstly, we prove that $(y_{n})$ is bounded. We argue by contradiction and assume that, up to a subsequence, $|y_{n}|\rightarrow \infty$ as $n\rightarrow \infty$. Take $R>0$ such that $\Lambda \subset B_{R}$. Since we may suppose that  $|y_{n}|>2R$, we have that for any $z\in B_{R/\e_{n}}$ 
$$
|\e_{n}z+y_{n}|\geq |y_{n}|-|\e_{n}z|>R.
$$
Taking into account $(u_{n})\subset \N_{\e_{n}}$, $(V_{1})$, Lemma \ref{DI}, the above inequality,  the definition of $\tilde{f}$, $v_{n}\rightarrow v$ in $H^{s}(\R^{N}, \R)$, and using the change of variable $x\mapsto z+\tilde{y}_{n}$ we get 
\begin{align*}
[v_{n}]^{2}+\int_{\R^{N}} V_{0} v_{n}^{2}\, dx &\leq \int_{\R^{N}} g(\e_{n} x+y_{n}, |v_{n}|^{2}) |v_{n}|^{2} \, dx \nonumber\\
&\leq \int_{B_{\frac{R}{\e_{n}}}} \tilde{f}(|v_{n}|^{2}) |v_{n}|^{2} \, dx+\int_{\R^{N}\setminus B_{\frac{R}{\e_{n}}}} f(|v_{n}|^{2}) |v_{n}|^{2}+|v_{n}|^{\2} \, dx \\
&\leq \frac{V_{0}}{k}\int_{\R^{N}} |v_{n}|^{2}\, dx+o_{n}(1),
\end{align*}
which implies that
$$
\min \left \{ 1, V_{0}\left(1-\frac{1}{k}\right) \right\}  \left([v_{n}]^{2}+\int_{\R^{N}} |v_{n}|^{2}\, dx\right)=o_{n}(1),
$$
that is $v_{n}\rightarrow 0$ in $H^{s}(\R^{N}, \R)$ and this is impossible. Therefore, $(y_{n})$ is bounded and we may assume that $y_{n}\rightarrow y_{0}\in \R^{N}$. It is obvious that the above arguments show that $y_{0}\in \overline{\Lambda}$. Finally we prove that $V(y_{0})=V_{0}$. Otherwise, if $V(y_{0})>V_{0}$, we can use (\ref{elenacritico}), Fatou's Lemma, the invariance of  $\R^{N}$ by translations and Lemma \ref{DI} to deduce that 
\begin{align*}
c_{0}=J_{0}(\tilde{v})&<\frac{1}{2}[\tilde{v}]^{2}+\frac{1}{2}\int_{\R^{N}} V(y_{0})\tilde{v}^{2} \, dx-\frac{1}{2}\int_{\R^{N}} F(|\tilde{v}|^{2})\, dx-\frac{1}{\2}\int_{\R^{N}} |\tilde{v}|^{\2}\, dx\\
&\leq \liminf_{n\rightarrow \infty}\left[\frac{1}{2}[\tilde{v}_{n}]^{2}+\frac{1}{2}\int_{\R^{N}} V(\e_{n}x+y_{n}) |\tilde{v}_{n}|^{2} \, dx-\frac{1}{2}\int_{\R^{N}} F(|\tilde{v}_{n}|^{2})\, dx -\frac{1}{\2}\int_{\R^{N}} |\tilde{v}_{n}|^{\2}\, dx \right] \\
&\leq \liminf_{n\rightarrow \infty}\left[\frac{t_{n}^{2}}{2}[|u_{n}|]^{2}+\frac{t_{n}^{2}}{2}\int_{\R^{N}} V(\e_{n}z) |u_{n}|^{2} \, dz-\frac{1}{2}\int_{\R^{N}} F(|t_{n} u_{n}|^{2})\, dz-\frac{1}{\2}\int_{\R^{N}} |t_{n}u_{n}|^{\2}\, dz  \right] \\
&\leq \liminf_{n\rightarrow \infty} J_{\e_{n}}(t_{n} u_{n}) \leq \liminf_{n\rightarrow \infty} J_{\e_{n}}(u_{n})\leq c_{0}
\end{align*}
which is impossible.
\end{proof}

\begin{proof}[Proof of Theorem \ref{thm2}]
Since the proof of Lemma \ref{moser} also works in the critical case, the only differences between the proofs of Theorem \ref{thm2} and Theorem \ref{thm1} consist in replacing Lemma \ref{prop3.3} and \eqref{4.18HZ} by Lemma \ref{prop3.3critico} and
\begin{align*}
g_{\e}(x, t^{2})t^{2}=f(t^{2})t^{2}+a^{\2}\leq \frac{V_{0}}{k}t^{2}, \mbox{ for all } x\in \R^{N}, |t|\leq \gamma,
\end{align*}
respectively.
\end{proof}

\section{supercritical magnetic problem}
\noindent
In this last section we study the following supercritical fractional problem
\begin{equation}\label{Rsupercritico}
(-\Delta)^{s}_{A_{\e}} u + V_{\e}(x)u =   |u|^{q-2}u+\lambda |u|^{r-2}u \mbox{ in } \R^{N}.
\end{equation}
Motivated by \cite{CY, R}, we truncate the nonlinearity $f(u)=|u|^{q-2}u+\lambda |u|^{r-2}u$ as follows.\\
Let $K>0$ be a real number, whose value will be fixed later, and we set
$$
f_{\lambda}(t):=
\begin{cases}
0 & \text{ if $t\leq 0$} \\
t^{\frac{q-2}{2}}+\lambda t^{\frac{r-2}{2}}& \text{ if $0<t<K$} \\
(1+\lambda K^{\frac{r-q}{2}})t^{\frac{q-2}{2}}   & \text{ if $t \geq K$}.
\end{cases}
$$
Then, it is easy to check that $f_{\lambda}$ satisfies assumptions $(f_1)$-$(f_4)$ ($(f_3)$ holds with $\theta=q>2$).\\
In particular
\begin{equation}\label{fk}
f_{\lambda}(t)\leq (1+\lambda K^{\frac{r-q}{2}})t^{\frac{q-2}{2}} \mbox{ for all } t\geq 0.
\end{equation}
Now we consider the following truncated problem
\begin{equation}\label{TR}
(-\Delta)^{s}_{A_{\e}} u + V_{\e}(x)u =  f_{\lambda}(|u|^{2})u \mbox{ in } \R^{N}, 
\end{equation}
and the corresponding functional $J_{\e, \lambda}: \h\rightarrow \R$ defined as
$$
J_{\e, \lambda}(u)=\frac{1}{2} \|u\|_{\e}^{2}-\frac{1}{2}\int_{\R^{N}} F_{\lambda}(|u|^{2})\, dx.
$$
We also introduce the autonomous functional $I_{0, \lambda}: H^{s}(\R^{N}, \R)\rightarrow \R$ given by
$$
I_{0, \lambda}(u)=\frac{1}{2} \|u\|_{0}^{2}-\frac{1}{2}\int_{\R^{N}} F_{\lambda}(u^{2})\, dx.
$$
Using Theorem \ref{thm1}, we know that for any $\lambda\geq 0$ there exists $\bar{\e}(\lambda)>0$ such that, for any  $\e\in (0, \bar{\e}(\lambda))$, problem (\ref{TR}) admits a nontrivial solution $u_{\e, \lambda}$.\\
Next we prove an auxiliary result which shows that the $H^{s}_{\e}$-norm of $u_{\e, \lambda}$ can be estimated from above by a constant independent of $\lambda$.
\begin{lem}\label{Fig1}
There exists $\bar{C}>0$ such that $\|u_{\e, \lambda}\|_{\e}\leq \bar{C}$ for any $\e>0$ sufficiently small.
\end{lem}
\begin{proof}
From the proof of Theorem \ref{thm1}, we know that any solution $u_{\e, \lambda}$ of (\ref{TR}) satisfies the following inequality
$$
J_{\e, \lambda}(u_{\e, \lambda})\leq c_{0, \lambda}+h_{\lambda}(\e)
$$
where $c_{0, \lambda}$ is the mountain pass level related to the functional $I_{0, \lambda}$ and $h_{\lambda}(\e)\rightarrow 0$ as $\e\rightarrow 0$.
Then, decreasing $\bar{\e}(\lambda)$ if necessary, we may assume that 
\begin{equation}
J_{\e, \lambda}(u_{\e, \lambda})\leq c_{0, \lambda}+1
\end{equation}
for any $\e\in (0, \bar{\e}(\lambda))$. Since $c_{0, \lambda}\leq c_{0, 0}$
for any $\lambda\geq 0$, we can infer that 
\begin{equation}\label{F1}
J_{\e, \lambda}(u_{\e, \lambda})\leq c_{0, 0}+1
\end{equation}
for any $\e\in (0, \bar{\e}(\lambda))$.
Moreover, using $(f_3)$, we can see that
\begin{align}\label{F2}
J_{\e, \lambda}(u_{\e, \lambda})&=J_{\e, \lambda}(u_{\e, \lambda})-\frac{1}{\theta}\langle J'_{\e, \lambda}(u_{\e, \lambda}), u_{\e, \lambda}\rangle \nonumber\\
&=\left(\frac{1}{2}-\frac{1}{\theta}\right)\|u_{\e, \lambda}\|^{2}_{\e}+\frac{1}{\theta}\int_{\R^{N}}  f_{\lambda}(|u_{\e, \lambda}|^{2})|u_{\e, \lambda}|^{2}-\frac{\theta}{2}F_{\lambda}(|u_{\e, \lambda}|^{2})\, dx \nonumber\\
&\geq  \left(\frac{1}{2}-\frac{1}{\theta}\right)\|u_{\e, \lambda}\|^{2}_{\e}.
\end{align}
Therefore (\ref{F1}) and (\ref{F2}) yield
$$
\|u_{\e, \lambda}\|_{\e}\leq \left[\left(\frac{2\theta}{\theta-2} \right)(c_{0, 0}+1)\right]^{\frac{1}{2}}=:\bar{C} \quad \forall \e\in (0, \bar{\e}(\lambda)).
$$
\end{proof}

\noindent
At this point we aim to prove that $u_{\e, \lambda}$ is a solution of the original problem (\ref{Psupercritico}). In order to achieve our purpose, we will show that we can find $K_{0}>0$ such that for any $K\geq K_{0}$, there exists $\lambda_{0}=\lambda_{0}(K)>0$ such that 
\begin{equation}
\||u_{\e, \lambda}|\|_{L^{\infty}(\R^{N})}\leq K \mbox{ for all }  \lambda\in [0, \lambda_{0}].
\end{equation}
In what follows we use a Moser iteration argument \cite{Moser} (see also \cite{CY, FF, R}).
For simplicity we will write $u$ instead of $u_{\e, \lambda}$.
\begin{proof}[Proof of Theorem \ref{thm3}]
For any $L>0$, we define $u_{L}:=\min\{|u|, L\}\geq 0$, where $\beta>1$ will be chosen later, and let $w_{L}=|u| u_{L}^{\beta-1}$. 
Taking $u_{L}^{2(\beta-1)}u$ in (\ref{TR}) we can see that
\begin{align}\label{conto1F}
&\Re\left(\iint_{\R^{2N}} \frac{(u(x)-u(y)e^{\imath A_{\e}(\frac{x+y}{2})\cdot (x-y)})}{|x-y|^{N+2s}} \overline{(u(x)u_{L}^{2(\beta-1)}(x)-u(y)u_{L}^{2(\beta-1)}(y)e^{\imath A_{\e}(\frac{x+y}{2})\cdot (x-y)})} \, dx dy\right)   \nonumber \\
&=\int_{\R^{N}} f_{\lambda}(|u|^{2}) |u|^{2}u_{L}^{2(\beta-1)}  \,dx-\int_{\R^{N}} V_{\e}(x) |u|^{2} u_{L}^{2(\beta-1)} \, dx.
\end{align}
Putting together (\ref{conto1F}), (\ref{fk}) and $(V_1)$ we get
\begin{align}\label{conto2F}
&\Re\left(\iint_{\R^{2N}} \frac{(u(x)-u(y)e^{\imath A_{\e}(\frac{x+y}{2})\cdot (x-y)})}{|x-y|^{N+2s}} \overline{(u(x)u_{L}^{2(\beta-1)}(x)-u(y)u_{L}^{2(\beta-1)}(y)e^{\imath A_{\e}(\frac{x+y}{2})\cdot (x-y)})} \, dx dy\right) \nonumber\\
&\leq C_{\lambda, K} \int_{\R^{N}} |v|^{q} v_{L}^{2(\beta-1)} \, dx
\end{align}
where $C_{\lambda, K}:=1+\lambda K^{\frac{r-q}{2}}$.
Arguing as in the first part of Lemma \ref{moser} we can see that
\begin{align*}
&\Re\left[(u(x)-u(y)e^{\imath A_{\e}(\frac{x+y}{2})\cdot (x-y)})\overline{(u(x)u_{L}^{2(\beta-1)}(x)-u(y)u_{L}^{2(\beta-1)}(y)e^{\imath A_{\e}(\frac{x+y}{2})\cdot (x-y)})}\right] \\
&\geq (|u(x)|-|u(y)|)(|u(x)|u_{L}^{2(\beta-1)}(x)-|u(y)|u_{L}^{2(\beta-1)}(y))
\end{align*}
which gives
\begin{align}\label{reale}
&\Re\left(\iint_{\R^{2N}} \frac{(u(x)-u(y)e^{\imath A_{\e}(\frac{x+y}{2})\cdot (x-y)})}{|x-y|^{N+2s}} \overline{(u(x)u_{L}^{2(\beta-1)}(x)-u(y)u_{L}^{2(\beta-1)}(y)e^{\imath A_{\e}(\frac{x+y}{2})\cdot (x-y)})} \, dx dy\right) \nonumber\\
&\geq \iint_{\R^{2N}} \frac{(|u(x)|-|u(y)|)}{|x-y|^{N+2s}} (|u(x)|u_{L}^{2(\beta-1)}(x)-|u(y)|u_{L}^{2(\beta-1)}(y))\, dx dy.
\end{align}
From formulas \eqref{Gg1} and \eqref{SS1} we get the following estimate
\begin{align}\label{conto3F}\begin{split}
\|w_{L}\|_{L^{2^{*}_{s}}(\R^{N})}^{2}&\leq C_{0} \beta^{2} \iint_{\R^{2N}} \frac{(|u(x)|-|u(y)|)}{|x-y|^{N+2s}} (|u(x)|u_{L}^{2(\beta-1)}(x)-|u(y)|u_{L}^{2(\beta-1)}(y))\, dx dy.
\end{split}\end{align}
Taking into account (\ref{conto2F}), \eqref{reale} and (\ref{conto3F}), and using the H\"older inequality we deduce that
\begin{align}\label{conto4F}
\|w_{L}\|_{L^{2^{*}_{s}}(\R^{N})}^{2}\leq C_{1} \beta^{2} C_{\lambda, K} \left(\int_{\R^{N}} |u|^{2^{*}_{s}}\, dx\right)^{\frac{q-2}{2^{*}_{s}}} \left(\int_{\R^{N}} w_{L}^{\frac{2 2^{*}_{s}}{2^{*}_{s}-(q-2)}} \, dx\right)^{\frac{2^{*}_{s}-(q-2)}{2^{*}_{s}}}
\end{align}
where  $2<\frac{2 2^{*}_{s}}{2^{*}_{s}-(q-2)}<2^{*}_{s}$ and $C_{1}>0$.
Then, using $\mathcal{D}^{s,2}(\R^{N}, \R)\subset L^{2^{*}_{s}}(\R^{N}, \R)$, Lemma \ref{DI} and Lemma \ref{Fig1}, we obtain
\begin{align}\label{conto4F}
\|w_{L}\|_{L^{2^{*}_{s}}(\R^{N})}^{2}\leq C_{2} \beta^{2} C_{\lambda, K} \bar{C}^{\frac{q-2}{2^{*}_{s}}} \|w_{L}\|_{L^{\alpha^{*}_{s}}(\R^{N})}^{2} 
\end{align}
where 
$$
\alpha^{*}_{s}:=\frac{2 2^{*}_{s}}{2^{*}_{s}-(q-2)}.
$$ 
Let us note that, if $|u|^{\beta}\in L^{\alpha^{*}_{s}}(\R^{N}, \R)$, the definition of $w_{L}$, $u_{L}\leq |u|$ and  (\ref{conto4F}) imply that
\begin{align}\label{conto5F}
\|w_{L}\|_{L^{2^{*}_{s}}(\R^{N})}^{2}\leq C_{3} \beta^{2} C_{\lambda, K} \bar{C}^{\frac{q-2}{2^{*}_{s}}} \left(\int_{\R^{N}} |u|^{\beta \alpha^{*}_{s}}\, dx\right)^{\frac{2}{\alpha^{*}_{s}}}<\infty.
\end{align}
Taking the limit  as $L \rightarrow \infty$ in (\ref{conto5F}) and using the Fatou Lemma we have
\begin{align}\label{conto6F}
\||u|\|_{L^{\beta 2^{*}_{s}}(\R^{N})}\leq (C_{4} C_{\lambda, K})^{\frac{1}{2\beta}} \beta^{\frac{1}{\beta}} \||u|\|_{L^{\beta \alpha^{*}_{s}}(\R^{N})}
\end{align}
provided that $|u|^{\beta \alpha^{*}_{s}}\in L^{1}(\R^{N}, \R)$.\\
Set $\beta:=\frac{2^{*}_{s}}{\alpha^{*}_{s}}>1$ and we note that, since $|u|\in L^{2^{*}_{s}}(\R^{N}, \R)$, the above inequality holds for this choice of $\beta$. Then, observing that $\beta^{2}\alpha^{*}_{s}=\beta 2^{*}_{s}$, it follows that \eqref{conto6F} holds with $\beta$ replaced by $\beta^{2}$, so we have
\begin{align*}
\||u|\|_{L^{\beta^{2} 2^{*}_{s}}(\R^{N})}\leq (C_{4} C_{\lambda, K})^{\frac{1}{2\beta^2}} \beta^{\frac{2}{\beta^{2}}}  \||u|\|_{L^{\beta^{2} \alpha^{*}_{s}}(\R^{N})}\leq  (C_{4} C_{\lambda, K})^{\frac{1}{2}\left(\frac{1}{\beta}+\frac{1}{\beta^{2}}\right)} \beta^{\frac{1}{\beta}+\frac{2}{\beta^{2}}} \||u|\|_{L^{\beta \alpha^{*}_{s}}(\R^{N})}.
\end{align*}
Iterating this process and using the fact that $\beta \alpha^{*}_{s}:=2^{*}_{s}$ we deduce that for every $m\in \mathbb{N}$
\begin{align}\label{conto7F}
 \||u|\|_{L^{\beta^{m} 2^{*}_{s}}(\R^{N})} \leq (C_{4} C_{\lambda, K})^{\sum_{j=1}^{m}\frac{1}{2\beta^{j}}} \beta^{\sum_{j=1}^{m} j\beta^{-j}} \||u|\|_{L^{2^{*}_{s}}(\R^{N})}.
\end{align}
Taking the limit as $m \rightarrow \infty$ in (\ref{conto7F}) and using the embedding $\mathcal{D}^{s,2}(\R^{N}, \R)\subset L^{2^{*}_{s}}(\R^{N}, \R)$, Lemma \ref{DI} and Lemma \ref{Fig1} we obtain
\begin{align}\label{conto9F}
\||u|\|_{L^{\infty}(\R^{N})}\leq (C_{4} C_{\lambda, K})^{\gamma_{1}} \beta^{\gamma_{2}} C_{5}
\end{align}
where $C_{5}= S_{*}^{-\frac{1}{2}} \bar{C}$, and
$$
\gamma_{1}:=\frac{1}{2}\sum_{j=1}^{\infty}\frac{1}{\beta^{j}}<\infty \quad \mbox{ and } \quad \gamma_{2}:=\sum_{j=1}^{\infty}\frac{j}{\beta^{j}}<\infty.
$$
Next, we will find suitable values of $K$ and $\lambda$ such that the following inequality holds
$$
(C_{4} C_{\lambda, K})^{\gamma_{1}} \beta^{\gamma_{2}} C_{5}\leq K,
$$
or equivalently
$$
1+\lambda K^{\frac{r-q}{2}}\leq C_{4}^{-1} \beta^{-\frac{\gamma_{2}}{\gamma_{1}}} (K C_{5}^{-1})^{\frac{1}{\gamma_{1}}}.
$$
Take $K>0$ such that 
$$
\frac{(K C_{5}^{-1})^{\frac{1}{\gamma_{1}}}}{C_{4}\beta^{\frac{\gamma_{2}}{\gamma_{1}}}}-1>0
$$
and fix $\lambda_{0}>0$ such that 
$$
\lambda\leq \lambda_{0}\leq \left[\frac{(K C_{5}^{-1})^{\frac{1}{\gamma_{1}}}}{C_{4}\beta^{\frac{\gamma_{2}}{\gamma_{1}}}}-1\right] \frac{1}{K^{\frac{r-q}{2}}}.
$$
Then, using (\ref{conto9F}) we can conclude that 
$$
\||u|\|_{L^{\infty}(\R^{N})}\leq K \mbox{ for all }  \lambda\in [0, \lambda_{0}].
$$
\end{proof}

\section*{Acknowledgements}
The author would like to thank the anonymous referee for her/his careful reading of the manuscript and valuable suggestions that improved the presentation of the paper.


\begin{thebibliography}{99}


\bibitem{AF}
C.O. Alves and G.M. Figueiredo,
{\it Multiple Solutions for a Semilinear Elliptic Equation with Critical Growth and Magnetic Field}, 
Milan J. Math., {\bf 82} (2) (2014), 389--405.
	
\bibitem{AFF}	
C.O. Alves, G.M. Figueiredo and M.F. Furtado, 
{\it Multiple solutions for a nonlinear Schr\"odinger equation with magnetic fields},
Comm. Partial Differential Equations {\bf 36} (2011), 1565--1586.	

\bibitem{AM}
C.O. Alves and O.H. Miyagaki,
{\it Existence and concentration of solution for a class of fractional elliptic equation in $\R^{N}$ via penalization method},
Calc. Var. Partial Differential Equations {\bf 55} (2016), art. 47, 19 pp.


\bibitem{AR}
A. Ambrosetti and P.H. Rabinowitz,
{\it Dual variational methods in critical point theory and applications}, 
J. Funct. Anal. {\bf 14} (1973), 349--381.


\bibitem{A1}
V. Ambrosio,
{\it Multiplicity of positive solutions for a class of fractional Schr\"odinger equations via penalization method},
Ann. Mat. Pura Appl. (4) {\bf 196} (2017), no. 6, 2043--2062.

\bibitem{Aade}
V. Ambrosio,
{\it Mountain pass solutions for the fractional Berestycki-Lions problem},
 Adv. Differential Equations {\bf 23} (2018), no. 5-6, 455--488-

\bibitem{A3}
V. Ambrosio, 
{\it Concentration phenomena for critical fractional Schr\"odinger systems}, 
 Commun. Pure Appl. Anal. {\bf 17} (2018), no. 5, 2085--2123.


\bibitem{A6}
V. Ambrosio,
{\it Concentrating solutions for a class of nonlinear fractional Schr\"odinger equations in $\R^{N}$},
to appear in Rev. Mat. Iberoam., preprint arXiv: \href{https://arxiv.org/abs/1612.02388}{1612.02388}.

\bibitem{AD}
V. Ambrosio and P. d'Avenia,
{\it Nonlinear fractional magnetic Schr\"odinger equation: existence and multiplicity},
 J. Differential Equations {\bf 264} (2018), no. 5, 3336--3368.

\bibitem{AS}
G. Arioli and A. Szulkin,
{\it A semilinear Schr\"odinger equation in the presence of a magnetic field},
Arch. Ration. Mech. Anal. {\bf 170} (2003), 277--295.
	
\bibitem{AHS}	
J. Avron, I. Herbst, B. Simon, 
{\it Schr\"odinger operators with magnetic fields. I. General interactions}, 
Duke Math. J. {\bf 45} (1978), no. 4, 847--883. 	
	


\bibitem{CS}
L.A. Caffarelli and L. Silvestre,
{\it An extension problem related to the fractional Laplacian},
Comm. Partial Differential Equations {\bf 32} (2007), 1245--1260.

\bibitem{CY}
J. Chabrowski and J. Yang,
{\it Existence theorems for elliptic equations involving supercritical Sobolev exponent},
Adv. Differential Equations {\bf 2} (1997), 231--256.	

\bibitem{Cingolani}
S. Cingolani,
{\it Semiclassical stationary states of nonlinear Schr\"odinger equations with an external magnetic field} 
J. Differential Equations {\bf 188} (2003), 52--79.

 



\bibitem{Cingolani-Secchi}
S. Cingolani and S. Secchi,
{\it Semiclassical states for NLS equations with magnetic potentials having polynomial growths},
J. Math. Phys. {\bf 46} (2005), 053503, 19 pp.



\bibitem{DS}
P. d'Avenia and M. Squassina,
{\it Ground states for fractional magnetic operators},
ESAIM Control Optim. Calc. Var. {\bf 24} (2018), no. 1, 1--24. 

\bibitem{DDPW}
J. D\'avila, M. del Pino, and J. Wei, 
{\it Concentrating standing waves for the fractional nonlinear Schr\"odinger equation}, 
J. Differential Equations {\bf 256} (2014), no. 2, 858--892.

\bibitem{DF}
M. del Pino and P. L. Felmer, 
{\it Local Mountain Pass for semilinear elliptic problems in unbounded domains},
Calc. Var. Partial Differential Equations, {\bf 4} (1996), 121--137.

\bibitem{DPV}
E. Di Nezza, G. Palatucci and E. Valdinoci, 
{\it Hitchhiker's guide to the fractional Sobolev spaces}, 
Bull. Sci. math. {\bf 136} (2012), 521--573.



\bibitem{DPMV}
S. Dipierro, M. Medina and E. Valdinoci, 
{\it Fractional elliptic problems with critical growth in the whole of $\R^{n}$}, 
Lecture Notes (Scuola Normale Superiore di Pisa) 15, Edizioni della Normale, Pisa, 2017. viii+152 pp.



\bibitem{EL}
M. Esteban and P.L. Lions,
{\it Stationary solutions of nonlinear Schr\"odinger equations with an external magnetic field}, 
Partial differential equations and the calculus of variations, Vol. I, 401--449, Progr. Nonlinear Differential Equations Appl., 1, Birkh\"auser Boston, Boston, MA, 1989.



\bibitem{FQT}
P. Felmer, A. Quaas and J.Tan,
{\it Positive solutions of the nonlinear {S}chr{\"o}dinger equation with the fractional Laplacian},
Proc. Roy. Soc. Edinburgh Sect. A {\bf 142} (2012), 1237--1262.

\bibitem{FF}
G. M. Figueiredo and M. Furtado, 
{\it Positive solutions for some quasilinear equations with critical and supercritical growth}, 
Nonlinear Anal. {\bf 66} (2007), no. 7, 1600--1616.


\bibitem{FigS}
G.M. Figueiredo and G. Siciliano,
{\it A multiplicity result via Ljusternick-Schnirelmann category and Morse theory for a fractional Schr\"odinger equation in $\R^{N}$},
NoDEA Nonlinear Differential Equations Appl. {\bf 23} (2016), art. 12, 22 pp.

\bibitem{FPV}
A. Fiscella, A. Pinamonti and E. Vecchi,
{\it Multiplicity results for magnetic fractional problems},
J. Differential Equations {\bf 263} (2017), 4617--4633. 



\bibitem{HZ}
X. He and W. Zou,
{\it Existence and concentration result for the fractional Schr\"odinger equations with critical nonlinearities},
Calc. Var. Partial Differential Equations {\bf 55} (2016), art. 91, 39 pp.

\bibitem{HIL}
F. Hiroshima, T. Ichinose and J. L\"orinczi,
{\it Kato's Inequality for Magnetic Relativistic Schr\"odinger Operators},
 Publ. Res. Inst. Math. Sci. {\bf 53} (2017), no. 1, 79--117.


\bibitem{I10}
T. Ichinose,
{\it Magnetic relativistic Schr\"odinger operators and imaginary-time path integrals},
Mathematical physics, spectral theory and stochastic analysis, 247--297, Oper. Theory Adv. Appl. {\bf 232}, Birkh\"auser/Springer Basel AG, Basel, 2013.


\bibitem{Kato}
T. Kato,
{\it Schr\"odinger operators with singular potentials}, 
Proceedings of the International Symposium on Partial Differential Equations and the Geometry of Normed Linear Spaces (Jerusalem, 1972), Israel J. Math. {\bf 13}, 135--148 (1973).


\bibitem{K}
K. Kurata,
{\it Existence and semi-classical limit of the least energy solution to a nonlinear Schr\"odinger equation with electromagnetic fields}, 
Nonlinear Anal. {\bf 41} (2000), 763--778.

\bibitem{Laskin1}
N. Laskin,
{\it Fractional quantum mechanics and L\`evy path integrals}, 
Phys. Lett. A {\bf 268} (2000), 298--305.



\bibitem{MBRS}
G. Molica Bisci, V. R\u{a}dulescu and R. Servadei,
{\it Variational Methods for Nonlocal Fractional Problems},
{\em Cambridge University Press}, \textbf{162} Cambridge, 2016.

\bibitem{Moser}
J. Moser,
{\it A new proof of De Giorgi's theorem concerning the regularity problem for elliptic differential equations},
Comm. Pure Appl. Math. {\bf 13} (1960), 457--468.

\bibitem{PP}
G. Palatucci and A. Pisante,
{\it Improved Sobolev embeddings, profile decomposition, and concentration-compactness for fractional Sobolev spaces}, 
Calc. Var. Partial Differential Equations {\bf 50} (2014), 799--829.

\bibitem{PSV2}
A. Pinamonti, M. Squassina and E. Vecchi,
{\it Magnetic BV functions and the Bourgain-Brezis-Mironescu formula}, 
to appear on Advances in Calculus of Variations, Preprint. arXiv:1609.09714. 

\bibitem{PSV1}
A. Pinamonti, M. Squassina and E. Vecchi,
{\it The Maz'ya-Shaposhnikova limit in the magnetic setting},
J. Math. Anal. Appl. {\bf 449} (2017), 1152--1159.


\bibitem{R}
P. H. Rabinowitz, 
{\it Variational methods for nonlinear elliptic eigenvalue problems}, 
Indiana Univ. Math. J. {\bf 23} (1973/74) 729--754.

\bibitem{Rab}
P. Rabinowitz,
{\it On a class of nonlinear Schr\"odinger equations}
Z. Angew. Math. Phys. {\bf 43} (1992), 270--291.

\bibitem{RS}
M. Reed and B. Simon, 
{\it Methods of Modern Mathematical Physics, I, Functional analysis},
Academic Press, Inc., New York, 1980. 

\bibitem{Secchi}
S. Secchi,
{\it Ground state solutions for nonlinear fractional Schr\"odinger equations in $\R^{N}$},
J. Math. Phys. {\bf 54} (2013), 031501.




\bibitem{W}
M. Willem,
{\it Minimax theorems},
Progress in Nonlinear Differential Equations and their Applications 24, Birkh\"auser Boston, Inc., Boston, MA, 1996.

\bibitem{ZSZ}
B. Zhang, M. Squassina and X. Zhang, 
{\it Fractional NLS equations with magnetic field, critical frequency and critical growth},
 Manuscripta Math. {\bf 155} (2018), no. 1-2, 115--140. 

\end{thebibliography}
\end{document}